\documentclass{amsart}

\usepackage{tilmansdef}
\usepackage{tikz}
\usepackage{xcolor}
\colorlet{darkgreen}{green!40!black}
\hypersetup{colorlinks=true,urlcolor=blue,citecolor=darkgreen,
linkcolor=darkgreen,linktocpage,bookmarksnumbered,unicode}

\usetikzlibrary{cd,graphs}
\DeclareMathOperator{\Hopf}{\mathcal H}
\DeclareMathOperator{\DMod}{DMod}
\DeclareMathOperator{\frob}{frob}
\newcommand{\DRing}{\mathcal R}
\newcommand{\HGraph}{\mathcal H\mathcal G}
\DeclareMathOperator{\height}{ht}
\DeclareMathOperator{\order}{o}
\DeclareMathOperator{\Sig}{Sig}
\DeclareMathOperator{\Str}{Str}

\title{On the structure of abelian Hopf algebras}
\author{Tilman Bauer}
\thanks{The author would like to thank the Mittag-Leffler Institute for supporting this research.} 
\date{July 2, 2024}

\begin{document}
\begin{abstract}
We study the structure of the category of graded, connected, commutative and cocommutative Hopf algebras over a perfect field $k$ of characteristic $p$. For $p$-torsion objects, we give a characterization of pure-injectives and a full classification of countably generated locally pure-injectives, extending the known classification of finitely generated objects. This classification is essentially equivalent to the classification of nonnegatively graded $k[s,t]/(st)$-modules, where $|s|=1$ and $|t|=-1$.
We obtain similar classifications of not necessarily $p$-torsion objects that are either free as commutative algebras or cofree as cocommutative coalgebras. 
\end{abstract}

\maketitle
\section{Introduction}

Building on work of Hopf \cite{hopf:Gruppenmannigfalten} and Borel \cite{borel:homologie-des-groupes-de-Lie}, Milnor and Moore's classical paper \cite{milnor-moore:hopf} described the structure of graded, connected, commutative Hopf algebras over a field $k$ \emph{as algebras} and a classification \emph{as Hopf algebras} in the case where they are primitively generated. In the simplest case, when the characteristic of $k$ is $0$ and the Hopf algebras in question are abelian (i.e., commutative and cocommutative), this is the end of the story. In that case, the primitives and indecomposables coincide and thus $H$ is primitively generated. A classification in positive characteristic is less approachable. Some years later, Schoeller \cite{schoeller:Hopf} observed that the category of graded abelian Hopf algebras over a perfect field $k$ of characteristic $p$ is abelian and thus, by the Freyd--Mitchell embedding theorem, isomorphic to the category of graded modules over a ring, which she described. This is Dieudonné theory in a graded setting. 

Although Dieudonné theory has been very successfully applied in various topological and algebraic contexts, it seems that the original question of classification has fallen into a bit of neglect. This paper aims to fix this to a certain extent.

\medskip

Let $k$ be a perfect field of characteristic $p$. Let $\Hopf$ be the category of nonnegatively graded, connected (i.e. $H_0=k$), abelian Hopf algebras over $k$. We will refer to objects of $\Hopf$ simply as ``Hopf algebras.''

Our main classification result concerns Hopf algebras $H$ that are $p$-torsion in the sense that the multiplication-by-$p$ map $[p]$ in the abelian group $\End(H)$ is trivial, and which satisfy the following condition, which uses the Frobenius ($p$th power) map $F\colon H_i \to H_{pi}$ and its dual, the Verschiebung $V\colon H_{pi} \to H_i$:

\begin{defn}
A $p$-torsion Hopf algebra $H$ over $k$ is \emph{locally pure-injective} if the following condition holds: For each sequence $(n_i)_{i \geq 0}$ of nonnegative numbers and each primitive element $x \in P(H)$, if a solution $\{x_i\}$ exists for every finite subset of the set of equations
\[
x_0 = x; \; V(x_i)=F^{n_i}x_{i-1}
\]
then a solution $\{x_i\}$ for the whole system exists.
\end{defn}

An illustrative special case to have in mind is when the Frobenius $F$ is trivial on $H$. Local pure injectivity in this case means that every element $x$ that is $V^n$-divisible for each $n \geq 0$ is the first element $x_0$ in a tower of $x_i$ with $Vx_i=x_{i-1}$. Equivalently, the first derived limit of the iterated Verschiebung is trivial:
\[
{\lim_V}^1 H=0.
\]
This derived limit description does not seem to easily generalize to the general case where $F\neq0$, though.

Note that any Hopf algebra of finite type is locally pure-injective, as is any injective Hopf algebra or any direct sum of locally pure-injective Hopf algebras. Ours is not the usual categorical definition of local pure injectivity, but it is equivalent to it, as we will show (Cor.~\ref{cor:lpicriterion}, Lemma~\ref{lemma:lpiequivalence}). It is a generalization of the concept of \emph{pure injectivity}, i.e. injectivity with respect to the class of maps that are filtered colimits of split injections (also known as \emph{pure maps}).

We offer the following classification result:

\begin{thm}\label{thm:modpclassification}
Let $H$ be a pure-injective $p$-torsion Hopf algebra over a perfect field of characteristic $p$.
\begin{enumerate}
	\item $H$ is indecomposable if and only if it is isomorphic to a Hopf algebra $H(r,m,n,I)$ described below, indexed by a tuple $(r,m,n,I)$ with $r \in \N$ coprime to $p$, $m \in \N_0$, $n \in \{m,m+1,\dots,\infty\}$, $I \subset \{m,\dots,n-1\}$ (resp. $I \subset \{m,m+1,\dots\}$ if $n=\infty$).
	\item $H$ is a retract of a product of indecomposable pure-injective Hopf algebras.
	\item $H$ is the pure-injective hull of a direct sum of indecomposable pure-injective Hopf algebras, i.e. of a sum of $H(r,m,n,I)$ for various tuples $(r,m,n,I)$.
\end{enumerate}
\end{thm}

If $H$ is countably generated, we can say more:
\begin{thm}\label{thm:modpcountableclassification}
Let $H$ be a countably generated, locally pure-injective $p$-torsion Hopf algebra over a perfect field of characteristic $p$. Then $H$ is isomorphic to a direct sum of indecomposable pure-injective Hopf algebras of the form $H(r,m,n,I)$.
\end{thm}

For Hopf algebras of finite type, these theorems were proven by Touzé \cite[Section 9]{touze:exponential} in the context of ``exponential functors,'' using results from \cite{crawley-boevey:string-algebras} on representations of string algebras. 

Similarly to the classification of finitely generated modules over PIDs, the decomposition in Theorem~\ref{thm:modpcountableclassification} is unique in the sense that it satisfies the Krull--Schmidt property: the unordered sequence of tuples $(r,m,n,I)$ occurring in the decomposition into indecomposables is the same for every such decomposition.

Explicitly, the Hopf algebra $H(r,m,n,I)$ is given as an algebra by
\[
H(r,m,n,I) = k[x_m,x_{m+1},\dots,x_n]/\Biggl(x_{i}^p - \begin{cases} x_{i+1}; &i \in I\\0;& i\geq n \text{ or }i \not\in I\end{cases}\Biggr), \quad |x_i| = rp^i,
\]
and it has a unique coalgebra structure that makes it indecomposable as a Hopf algebra. Note that $H(r,m,n,I)$ has dimension $1$ in degrees $0,rp^m,2rp^m,\dots,r(p^{n+1}-1)$ (resp. in degrees $rp^m,2rp^m,\dots$ if $n=\infty$) and is trivial in all other dimensions.

Hopf algebras that are $p$-torsion but not locally pure-injectivecan be more complicated; in particular, there are indecomposables that are not of finite type (Example~\ref{ex:nonpureinjective}). 

Similarly, a satisfying classification does not exist for general (non-$p$-torsion) Hopf algebras, even under a pure injectivity condition. There are indecomposable Hopf algebras of arbitrarily large dimension in any given degree (Example~\ref{ex:largeindecomposables}). The situation greatly improves if $H$ is free as a (commutative) algebra or cofree as a (cocommutative) coalgebra:

\begin{thm}\label{thm:freeorcofreeconcrete}
Given any tuple $(r,m,n,I)$ as in Theorem~\ref{thm:modpclassification}, there exist
\begin{itemize}
\item a unique Hopf algebra $H_f(r,m,n,I)$ which is free as an algebra, and 
\item a unique Hopf algebra $H_c(r,m,n,I)$ which is cofree as a coalgebra
\end{itemize}
such that $H_f(r,m,n,I)/[p] \cong H(r,m,n,I) \cong H_c(r,m,n,I)/[p]$. Any countably generated Hopf algebra which is free as an algebra or cofree as a coalgebra and whose mod-$[p]$ reduction is locally pure-injective, decomposes uniquely into a tensor product of Hopf algebras isomorphic to $H_f(r,m,n,I)$ (resp. $H_c(r,m,n,I)$).
\end{thm}

The condition of being free as an algebra is strictly weaker than being free over a coalgebra, or being a projective object in Hopf algebras.

In the general case, we can at least classify all Hopf algebras $H$ with a given indecomposable reduction modulo $[p]$. We define:

\begin{defn}
A \emph{basic Hopf graph} is a (finite or infinite) connected string of arrows with end points in $\N_0 \times \N_0$ with the following properties:
\begin{enumerate}
	\item Every arrow goes either from $(i,j)$ to $(i+1,j)$, from $(i+1,j)$ to $(i,j)$, from $(i,j+1)$ to $(i+1,j)$, or from $(i+1,j+1)$ to $(i,j)$.
	\item Every arrow shares each of its end points with exactly one other arrow, unless the end point lies on the $x$-axis (i.e. is of the form $(i,0)$), in which case it may share its end point with no other arrow.
\end{enumerate}
Denote by $\HGraph$ the set of basic Hopf graphs.
\end{defn}

Thus a basic Hopf graph looks something like this:

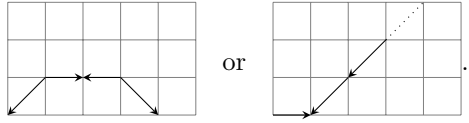
\[
\begin{matrix}
\begin{tikzpicture}[scale=0.5]
\draw[gray,very thin] (0,0) grid (5,3);
\draw[stealth-] (0,0) -- (1,1);
\draw[-stealth] (1,1) -- (2,1);
\draw[stealth-] (2,1) -- (3,1);
\draw[-stealth] (3,1) -- (4,0);
\draw[->](1,-1) -- (3,-1) node[midway,below]{$i$};
\draw[->](-1,1)-- (-1,3) node[midway,left]{$j$};
\end{tikzpicture}
\end{matrix} \qquad \text{or} \qquad
\begin{matrix}
\begin{tikzpicture}[scale=0.5]
\draw[gray,very thin] (0,0) grid (5,3);
\draw[-stealth] (1,0) -- (2,0);
\draw[stealth-] (2,0) -- (3,1);
\draw[stealth-] (3,1) -- (4,2);
\draw[dotted] (4,2) -- (5,3);
\draw[->](1,-1) -- (3,-1) node[midway,below]{$i$};
\draw[->](-1,1)-- (-1,3) node[midway,left]{$j$};
\end{tikzpicture}
\end{matrix}
\]

To any basic Hopf graph $\Gamma$ and natural number $r$ coprime to $p$, we can associate a triple $\sigma_\Gamma = (m_\Gamma,n_\Gamma,I_\Gamma)$ as in Theorem~\ref{thm:modpclassification}, where $m_\Gamma$ and $n_\Gamma$ are the left and right ends of the string, respectively, and $i \in I_\Gamma$ iff there is a right or right-down pointing arrow from $i$ to $i+1$ in $\Gamma$.

\begin{thm}\label{thm:basichopfclassification}
For every basic Hopf graph $\Gamma$ and $r$ coprime to $p$, there exists a unique Hopf algebra $H=H(r,\Gamma)$ such that $H/[p] \cong H(r,m_\Gamma,n_\Gamma,I_\Gamma)$ and with the properties:
\begin{enumerate} 
\item $V\colon H_{rp^i} \to H_{rp^{i-1}}$ is injective iff $\Gamma$ contains an arrow from $(i,j)$ to $(i-1,j)$ or an arrow from $(i-1,j)$ to $(i,j-1)$;
\item $F\colon H_{rp^{i-1}} \to H_{rp^i}$ is injective iff $\Gamma$ contains an arrow from $(i-1,j)$ to $(i,j)$ or an arrow from $(i,j)$ to $(i-1,j-1)$.
\end{enumerate}
Any Hopf algebra whose mod-$[p]$ reduction is isomorphic to some $H(r,m,n,I)$ is itself isomorphic to $H(r,\Gamma)$ for exactly one pair $(r,\Gamma)$. 
\end{thm}

Once translated into statements about Dieudonné modules, none of the methods and proofs in this paper use particularly sophisticated techniques.

An early attempt at classifying Hopf algebras was made in \cite{wraith:abelian-hopf-algebras}, but that paper contains some key errors that invalidate its main results, such as that the category of Hopf algebras does not have projectives or injectives (it does), that the homological dimension is $1$ (it is $2$ \cite{schoeller:Hopf}), and an incorrect classification of indecomposable objects. One such error is the assumption that two Hopf algebras which are isomorphic both as algebras and as coalgebras must be isomorphic as Hopf algebras \cite[4.9]{wraith:abelian-hopf-algebras}. A counterexample to this is Kuhn's ``cautionary example'' \cite[Example 1.13]{kuhn:quasi-shuffle}.

\begin{remark}
If $H$ is a \emph{graded-commutative} Hopf algebra and $\operatorname{char}(k)>2$ then there is a natural splitting $H = H^{\operatorname{even}} \otimes H^{\operatorname{odd}}$ \cite[Proposition~A.4]{bousfield:p-adic-lambda-rings} where $H^{\operatorname{even}}$ is concentrated in even degrees and $H^{\odd}$ is the exterior algebra on the odd-dimensional primitive elements of $H$. Thus odd Hopf algebras are simply classified by the underlying graded $k$-vector space of primitive elements. We therefore lose little generality by only considering commutative Hopf algebras (i.e. not graded-commutative).
\end{remark}

\subsection*{Acknowledgements} I would like to thank Antoine Touzé for making me aware of his prior proof of Theorems~\ref{thm:modpclassification}/\ref{thm:modpcountableclassification} in the finite-type case and the essential role string algebras play in it, and H. Krause and K. Schlegel for pointing out an error in an earlier version of this paper.

\section{\texorpdfstring{$p$-typical splittings and the Dieudonné equivalence}{p-typical splittings and the Dieudonné equivalence}}

Let $k$ be a perfect field of characteristic $p$. We recall the classical Dieudonné equivalence in the graded context \cite{schoeller:Hopf}. Denote by $W(k)$ the ring of $p$-typical Witt vectors of $k$ and by $\frob\colon W(k) \to W(k)$ its Frobenius, i.e. the unique lift of the $p$th power map on $k$ to a linear map on $W(k)$. Denote by $\DMod$ the category of Dieudonné modules, i.e. positively graded $W(k)$-modules together with additive operators $F\colon M_n \to M_{pn}$ and $V\colon M_{pn} \to M_n$, for each $n$, satisfying
\begin{itemize}
	\item $FV=VF=p$
	\item $F(\lambda x)=\frob(\lambda) F(x)$ for $\lambda \in W(k)$, $x \in M$;
	\item $V(\frob(\lambda) x) = \lambda V(x)$ for $\lambda \in W(k)$, $x \in M$.
\end{itemize}
By convention, $V(x)=0$ unless $p \mid |x|$.

\begin{thm}[{\cite{schoeller:Hopf}}] \label{thm:dieudonneeq}
There is an equivalence of abelian categories
\[
D\colon \Hopf \to \DMod.
\]
\end{thm}

Let us call a module $M \in \DMod$ \emph{$p$-typical of type $j$}, where $j$ is a natural number coprime to $p$, if $M_i=0$ unless $i=jp^a$ for some $a \geq 0$. Denote the full subcategory of $p$-typical Dieudonné modules of type $j$ by $\DMod^{(j)}$. Similarly, a  Hopf algebra $H$ is called \emph{$p$-typical} of type $j$ if $(QH)_i=0$ unless $i=jp^k$ for some $k$, where $QH$ denotes the module of indecomposables. Denote the full subcategory of $p$-typical Hopf algebras of type $j$ by $\Hopf^{(j)}$.

\begin{lemma}[{\cite[\textsection 2]{schoeller:Hopf},\cite{ravenel:dieudonne}}]\label{lemma:splitting}
A Hopf algebra $H$ is $p$-typical of type $j$ iff $D(H)$ is a $p$-typical Dieudonné module of type $j$. There are splittings of abelian categories
\[
\DMod \simeq \prod_{(j,p)=1} \DMod^{(j)} \quad \text{ and } \quad \Hopf \simeq \prod_{(j,p)=1} \Hopf^{(j)}
\]
compatible with the Dieudonné functor $D$. Finally, for any $j$ coprime to $p$, there is an equivalence of categories
\[
\Lambda^{(j)}\colon \DMod^{(j)} \xrightarrow{\simeq} \Mod_{\DRing},
\]
where $\DRing = W(k)[s,t]/(st-p)$ is a $\Z$-graded commutative ring with $|s|=1$, $|t|=-1$, and $\Mod_{\DRing}$ denotes the category of nonnegatively graded modules over it.
\end{lemma}
\begin{proof}
The preservation of $p$-typicality under the Dieudonné functor follows from its construction.
The $p$-typical splitting on the module level is obvious because $F$ and $V$ preserve the type $j$. The corresponding splitting of Hopf algebras follows.

For the last statement, construct a functor $\Lambda^{(j)}\colon \DMod^{(j)} \to \Mod_{\DRing}$ by
\[
\Lambda^{(j)}(M)_n = M_{jp^n} \quad \text{with $W(k)$-module structure given by } \lambda.x = \lambda^{p^n} x
\]
and define $s\colon \Lambda^{(j)}(M)_n \to \Lambda^{(j)}(M)_{n+1}$ by $s(x)=F(x)$ and similarly $t(x)=V(x)$. Since $k$ is perfect, $\frob$ is bijective on $W(k)$, and $\Lambda^{(j)}$ is an equivalence of categories.
\end{proof}

Using this, the classification of Hopf algebras is thus reduced to the classification of modules in $\Mod_{\DRing}$. We let
\begin{equation}\label{eq:ptypicaldieudonne}
D^{(j)} \colon \Hopf^{(j)} \xrightarrow{D} \DMod^{(j)} \xrightarrow{\Lambda^{(j)}} \Mod_{\DRing}
\end{equation}
denote the composite equivalence between $p$-typical Hopf algebras of type $j$ and modules over $\DRing$.

Similarly, the classification of $p$-torsion Hopf algebras is reduced to the category of modules in $\Mod_R$, where $R=\mathcal R/(p) = k[s,t]/(st)$.

\section{\texorpdfstring{The classification of graded $k[s,t]/(st)$-modules}{The classification of graded k[s,t]/(st)-modules}} \label{sec:kstmodules}

In this section, we will consider the category $\Mod_R$ of $\Z$-graded $R$-modules and restrict to the subcategory $\Mod_R^+$ of nonnegatively graded $R$-modules only when necessary.

\subsection{\texorpdfstring{Duality of $\mathcal R$-modules}{Duality of ℛ-modules}}

Given an $\mathcal R$-module $M$, we denote by $M^\op$ the (covariant) dual given by 
\[
(M^\op)_m = M_{-m}
\]
This becomes an $\mathcal R$-module by interchanging the $s$- and $t$-action, and descends to a duality of $R$-modules (recall that $R=\mathcal R/(p)$).

\begin{remark}
Note that this duality is not meaningful on the level of Hopf algebras since it does not send nonnegatively graded modules to nonnegatively graded modules. The duality that corresponds to Cartier duality (linear duals of Hopf algebras) is Pontryagin (or Matlis) duality $M^*=\Hom_{W(k)}(M,CW(k))$, which does not satisfy $M^{**}=M$ unless $M$ is of finite length. We will use the simple-minded covariant dual to save work in the sections that follow.
\end{remark}

\subsection{\texorpdfstring{Strings in $R$-modules}{Strings in R-modules}}

The fundamental concept in our study of $R$-modules is that of a string, which is a sequence of elements connected by either $s$- or $t$-multiplication.

\begin{defn}
A \emph{string} $C$ in $M \in \Mod_R$ is a sequence of elements $a_i \in M_i$ in consecutive degrees $i$ such that either $t(a_i)=a_{i-1}$ or $s(a_{i})=a_{i+1}$. We allow finite strings ($i \in \{m,\dots,n\}$) as well as infinite strings in either or both directions.

We write $\Str_{m,n}$ for the set of strings from degree $m$ to degree $n$, allowing $n$ to be $\infty$ and $m$ to be $-\infty$.

We write $[m,n]$ for the range $\{m,\dots,n\}$ when $-\infty<m\leq n<\infty$, $[m,\infty]=\{m,m+1,\dots\}$, $[-\infty,n] = \{\dots,n-1,n\}$ and $[-\infty,\infty]=\Z$.

A \emph{signature} is a tuple $\sigma=(m,n,I)$ where $-\infty\leq m \leq n\leq \infty$ and $I \subseteq [m,n-1]$.

We say that a string $C=(a_i)_{i \in [m',n']} \in \Str_{m',n'}$, for $m' \leq m$ and $n' \geq n$, \emph{adheres} to the signature $\sigma=(m,n,I)$ if:
\begin{enumerate}
	\item $a_i = 0$ unless $m \leq i \leq n$;
	\item $s(a_i)=a_{i+1}$ if $i\in I$;
	\item $t(a_{i+1})=a_i$ if $i \not\in I$.
\end{enumerate}
We write $\Str_\sigma$ for the $k$-vector space of strings adhering to the signature $\sigma$.
\end{defn}

Note that $\Str_{m,n}$ is not an abelian group since sums of strings of distinct signatures usually are not strings. A string has a unique signature if and only if all of its elements are nonzero.

\begin{defn}
We say that a string $C$ is \emph{strongly degenerate} if it contains a zero element, and \emph{degenerate} if it is a $k$-linear combination of strongly degenerate strings (of the same size, and adhering to a common signature).
\end{defn}

\begin{example}\label{ex:illustrativestringexample}
Consider the following $R$-module $M$:
\[
\tikz[>=stealth]\graph[grow up, branch right,  nodes={math nodes,inner sep=1.5pt},empty nodes] { a <- { b -> d; c }};
\]
In this and similar pictures, the grading is given by the height (the degree of $a$ is $0$, the degree of $b$ and $c$ is $1$, and the degree of $d$ is $2$). All letters are generators of a copy of $k$, downward arrows denote a nontrivial $t$ (thus $t(b)=t(c)=a$), and upward arrows denote a nontrivial $s$ (thus $s(b)=d$).

In $M$, the string $(a,b,d)$ is degenerate (but not strongly degenerate), as it is the sum of the degenerate strings $(0,b-c,d)$ and $(a,c,0)$. The strings $(a,b)$, $(a,c)$, and $(b,d)$ are all nondegenerate. 
\end{example}

\begin{defn}
We say that a string $(a_m,\dots,a_n)$ is \emph{lower closed} if $m=-\infty$ or $t(a_m)=0$, \emph{upper closed} if $n=\infty$ or $s(a_n)=0$, and \emph{closed} if it is both.
\end{defn}

Note that every string $C$ has a unique upper closure $\overline C$, lower closure $\underline C$, and closure $\overline{\underline C}$, given as the minimal (upper/lower) closed string containing $C$. Example~\ref{ex:illustrativestringexample} shows that the closure of a nondegenerate string may well become degenerate.

We introduce some more notation:
\begin{defn}$\;$
\begin{itemize}
	\item $\Str^+_m = \coprod_{m \leq n \leq \infty} \Str_{m,n}$ is the set of \emph{upward strings};
	\item $\Str^-_n = \coprod_{-\infty \leq m \leq n} \Str_{m,n}$ is the set of \emph{downward strings}. 
	\item $\Str^+(a)$ and $\Str^-(a)$ are the sets of upward and downward strings, respectively, with extremal element $a \in M$.
	\item $\Str^\pm_\sigma(a)$ is the set $\Str^\pm(a) \cap \Str_\sigma$, the set of upward/downward strings adhering to the signature $\sigma$ with extremal element $a$.
	\item For a string $C\in \Str_{m,n}$ and $m \leq m'\leq n'\leq n$, we denote by $C|_{m'}^{n'} \in \Str_{m',n'}$ the substring in degrees $m'$ through $n'$; we use the same notation for signatures, $\sigma|_{m'}^{n'}$.
\end{itemize}
\end{defn}

Finally, we consider modules generated by strings:
\begin{defn}
A submodule $M<N$ is called a \emph{string module} if it is freely generated by the elements of a string $C$ that is not strongly degenerate. Note that this implies that $C$ is closed.
For a signature $\sigma=(m,n,I)$, the \emph{string module associated to $\sigma$} is the unique module $M(\sigma)$ such that $M(\sigma)_i =k\langle x_i\rangle$ for $m \leq i \leq n$ and $M(\sigma)_i=0$ otherwise, and such that $(x_m,\dots,x_n)$ is a string of signature $\sigma$.
\end{defn}

\subsection{Degenerate strings}

Under $R$-linear maps, degenerate strings map to degenerate strings and thus preimages of nondegenerate strings are nondegenerate. Substrings of nondegenerate strings are nondegenerate, but concatenations of nondegenerate strings can become degenerate (cf. Ex.~\ref{ex:illustrativestringexample}).

A finite degenerate string can always be written as a sum of two strongly degenerate strings, one of which has a zero as its topmost element and the other one has a zero as its bottom element. Again, Ex.~\ref{ex:illustrativestringexample} illustrates this.

\begin{lemma}\label{lemma:degenerateupperlowerdecomp}
A finite string $C \in \Str_\sigma$ is degenerate iff $C=C^u + C^l$ for some $C^u \in \Str^+_\sigma(0)$ and some $C^l \in \Str^-_\sigma(0)$.
\end{lemma}
\begin{proof}
It is clear that $C$ is degenerate if it can be written as such a sum, so let us consider the converse. We use induction on $n-m$. The cases $n-m=0$ and $n-m=1$ are clear. Let $C=C_1+\cdots+C_k$ be a decomposition into strongly degenerate strings. Without loss of generality, suppose that $C_1,\dots,C_l$ have vanishing bottom element and $C_{l+1},\cdots,C_k$ have nonvanishing bottom element.

The substrings $C_i'=C_i|_{m+1}^n$ for $i=l+1,\dots,k$ exhibit the corresponding substring $C'=(C_{l+1}+\dots+C_k)|_{m+1}^n$ as degenerate. Thus $C' = (C')^u+(C')^l$ by inductive assumption. By setting $C^u$ to be the sum of the downward extension of $(C')^u$ with $0$ and $C_1+\dots+C_l$ and $C^l$ the downward extension of $(C')^l$ with the bottom element of $C$, we obtain the desired decomposition.
\end{proof}

Let us call a string $C \in \Str_\sigma$ \emph{nondegenerate-extended} if it contains a nondegenerate substring $C'$, and the complementary elements $C-C'$ are all zero.

\begin{corollary}\label{cor:generatedbynondegenerates}
Any finite string $C \in \Str_\sigma$ is a finite sum of nondegenerate-extended strings.
\end{corollary}
\begin{proof}
We use induction on the length of $C$, the base case of length $0$ being trivial. Similarly, the result is obvious if $C$ itself is nondegenerate.

If $C$ is degenerate, decompose $C=C^u + C^l$ as in Lemma~\ref{lemma:degenerateupperlowerdecomp} into two shorter strings and apply the induction hypothesis.
\end{proof}

\begin{lemma}\label{lemma:degenerateinfinitestringchar}
An (infinite) string $C$ is nondegenerate if and only if all finite substrings are nondegenerate.
\end{lemma}
\begin{proof}
Suppose $C$ is degenerate, with $C=C_1 + \cdots + C_k$ a decomposition into strongly degenerate summands. Suppose $C_i = (a_{ij})$ has $a_{i,j_i}=0$. Then the finite substring of $C$ contained within degrees $\min(j_i)$ and $\max(j_i)$ is degenerate.

For the converse direction, suppose that $C$ is a string and the finite substring $C|_m^n$ from degree $m$ to degree $n$ is degenerate. Let $C|_m^n = C^u+C^l$ be a decomposition as in Lemma~\ref{lemma:degenerateupperlowerdecomp}. Let $C_1 = C^u \cup C|_{n+1}^\infty$ and $C_2 = C^l \cup C|_{-\infty}^{m-1}$. These are degenerate and $C=C_1+C_2$, so $C$ is degenerate.
\end{proof}

\begin{corollary}\label{cor:degenerateupperlowerdecomp}
For any signature $\sigma=(m,n,I)$ with $-\infty \leq m \leq n \leq \infty$ and any degenerate string $C \in \Str_{\sigma}$, there are some finite integers $m'$, $n'$ with $m \leq m' \leq n' \leq n$ such that $C$ can be written as a sum $C=C^u+C^l$ of two strongly degenerate strings with $C^u$ zero in degrees $m'$ and below, and $C^l$ zero in degrees $n'$ and above.
\end{corollary}
\begin{proof}
By Lemma~\ref{lemma:degenerateinfinitestringchar}, $C$ contains a finite degenerate substring $C|_{m'}^{n'}$ adhering to $\sigma'=\sigma_{m'}^{n'}$. By Lemma~\ref{lemma:degenerateupperlowerdecomp}, $C|_{m'}^{n'} = C^u+C^l$ for some $C^u \in \Str^+_{\sigma'}(0)$ and $C^l \in \Str^-_{\sigma'}(0)$. Splicing $C_l$ with $C|_m^{m'}$ and $C^u$ with $C|_{n'}^n$ gives the desired decomposition.
\end{proof}

\subsection{A total order on signatures}

\begin{defn}
We denote the set of signatures with fixed $m$ and $n$ by $\Sig_{m,n}$, the sets of signatures with bottom degree $m$ by $\Sig^+(m)$, and the sets of signatures with top degree $n$ by $\Sig^-(n)$, in the same way as for strings.
\end{defn}

We define a total order on $\Sig^+(m)$ and $\Sig^-(n)$ as a variant of the lexicographic order. It is easiest to describe by embedding these sets into the real numbers.

\begin{defn} \label{def:realvalue}
The \emph{real value} $r(\sigma) \in \R$ of $\sigma = (m,n,I) \in \Sig^+(m)$ is the real number which in decimal expansion is given by
\[
0.d_md_{m+1}\dots
\]
where
\[
d_i = \begin{cases} 4; & i<n \text{ and } i \in I;\\
6; &i<n \text{ and } i \not\in I;\\
5; &i \geq n.
\end{cases}
\]
We define a total order on $\Sig^+(m)$ by $\sigma \leq \sigma' \iff r(\sigma) \leq r(\sigma')$ and call it the \emph{ascending order}.
\end{defn}

\begin{example}
We can depict a signature $\sigma=(m,n,I) \in \Sig^+(m)$ graphically by a string of $s$ and $t$, depending on whether $i \in I$ or not. 
The real value of $ttst$, with the left hand side being dimension $m$, is $0.6646555\dots$.
\end{example}

For any signatures $x$, $y$ and $z$, the ascending order satisfies
\[
(x s y) < (x) < (x t z).
\]

\begin{remark}
Of course, any other $n$-adic representation could be used as long as one does not use both the digit $0$ and the digit $n-1$ (because of the nonuniqueness of $n$-adic expansions). So $n$ has to be at least $4$.
\end{remark}

\begin{defn}
Similarly, we define a \emph{descending} order on $\Sig^-(n)$ by duality: if $\sigma=(m,n,I)$ is a signature then the dual signature is defined as
\[
\sigma^\op = (-n,-m,\{-i \mid i \not\in I,m\leq i \leq n\})
\]
If $\sigma \in \Sig^-(n)$ then $\sigma^\op \in \Sig^+(-n)$, and we define $\sigma<\sigma'$ iff $\sigma^\op > (\sigma')^\op$. 
\end{defn}

The duality is designed to be compatible with the covariant duality of modules: $C$ is a string in $M$ adhering to $\sigma$ if and only if the dual string $C^\op$ in $M^\op$ adheres to $\sigma^\op$.

For any signatures $x$, $y$ and $z$, the descending order satisfies
\[
(y s x) < (x) < (z t x).
\]

Note that the ascending order on $\Sig^+(m)$ is neither a well-ordering nor a co-well-ordering:
\[
()<(t)<(tt)<\cdots \quad \text{and} \quad ()>(s)>(ss) >\cdots.
\]
The same is true for the descending order unless one restricts to signatures of nonnegatively graded modules, in which case $\Sig^-(n)$ is finite.

The order on signatures is compatible with whether an element is the bottom (or top) element of some string of that signature:
\begin{lemma}\label{lemma:filtering}
Let $a \in M_m$, $m \leq n \leq \infty$, and $\sigma\geq\tau \in \Sig(m,n)$ be two signatures such that $\Str_\sigma^+(a)$ is nonempty. Then $\Str_\tau^+(a)$ is also nonempty. 

Similarly, for $b \in M_n$,  if $\Str_{\tau}^-(b)$ is nonempty then so is $\Str_{\sigma}^-(b)$.
\end{lemma}
\begin{proof}
Let $\sigma=(m,n,I)$ and $\tau = (m,n',J)$. Then there exists $k \geq m$ such that $\sigma|_m^k = \tau|_m^k$ and $k \in J-I$. Thus if $C = (a_i)_{m \leq i \leq n} \in \Str_\sigma^+(a)$ then $t(a_{k+1})=a_k$ and thus $s(a_k)=0$. Hence the string $(a_m,\dots,a_k,0,\dots,0)$ adheres to $\tau$.

The statement about downward strings follows from duality.
\end{proof}

\begin{defn}
We say that a string $C\in \Str_{m,n}$ is \emph{normalized} if there exists an $m' \in [m,n]$ such that $C|_m^{m'}$ is nondegenerate and $C|_{m'+1}^n = 0$.
\end{defn}

In particular, nondegenerate strings are automatically normalized.

\begin{lemma}\label{lemma:normalizedstrings}
Let $M$ be an $R$-module and $a \in M_m-\{0\}$.
\begin{enumerate}
	\item If $\Str^+_\sigma(a)$ is nonempty then it contains a normalized string $C$.
	\item If $C \in \Str^+_\sigma(a)$ is normalized then the signature of the nondegenerate part of $C'$ is greater than or equal to $\sigma$, and equal if and only if $C$ is nondegenerate. 
	\item If $C \in \Str^+_\sigma(a)$ contains any degenerate string then it containes a normalized, upper closed string.
\end{enumerate}
\end{lemma}

\begin{proof}
\begin{enumerate}
	\item 
We use induction on the length of the string $C$. If $C$ is nondegenerate, it is normalized and there is nothing to show. This happens in particular if $C$ has length $1$ (the minimal possible length since $a \neq 0$). Thus suppose $C$ is degenerate.

Let $\sigma=(m,n,I)$. Let $C=C^u+C^l$ as in Lemma~\ref{lemma:degenerateupperlowerdecomp}. Then $C^l \in \Str^+_{\sigma}(a)$. Let $C'=C^l|_m^{j}$ be the longest substring that is not strongly degenerate. Even if $C'$ is not yet normalized, there is, by induction, a normalized string $C'' \in \Str_{\sigma'}^+(a)$ of signature at least $\sigma'$, which can be extended with zeroes to a normalized string $C' \in \Str^+_{\sigma'}(a)$. This proves the case of finite strings $C$.

If $C$ is infinite and degenerate, it has a finite, degenerate, closed substring $C'$ by Lemma~\ref{lemma:degenerateinfinitestringchar}. Applying the finite case to $C'$ and extending by zeroes upward gives the result for $C$.
\item Let $C \in \Str^+(a)$ be normalized with nondegenerate part $\sigma'=\sigma|_m^j$. We necessarily must have $j \in I$. Thus $\sigma'\geq\sigma$ and $\sigma'=\sigma$ if and only if $j=n$.
\item If $C \in \Str^+_\sigma(a)$ is degenerate then use (1) to choose a normalized string $C'$, whose nondegenerate part by (2) has greater signature than $\sigma$, which means that it is zero in dimension $n$ and above, hence closed.
\end{enumerate}
\end{proof}

The following proposition says that a string is nondegenerate iff its signature is extremal:

\begin{prop}\label{prop:degeneratechar}
Let $\sigma=(m,n,I)$ be a signature, $a \in M_m$, $b \in M_n$, and $C \in \Str_\sigma$ a string between $a$ and $b$. Then the following are equivalent:
\begin{enumerate}
	\item $C$ is degenerate; \label{prop:degeneratechar:deg}
	\item There is an upper closed string $C' \in \Str_{m,n} \cap \Str^+(a)$ adhering to a bigger signature than $\sigma$; \label{prop:degeneratechar:closedupstring}
	\item There is a string $C' \in \Str_{m,n} \cap \Str^+(a)$ adhering to a bigger signature than $\sigma$; \label{prop:degeneratechar:upstring}
	\item There is a lower closed string $C' \in \Str_{m,n} \cap \Str^-(b)$ adhering to a smaller signature than $\sigma$; \label{prop:degeneratechar:closeddownstring}
	\item There is a string $C' \in \Str_{m,n} \cap \Str^-(b)$ adhering to a smaller signature than $\sigma$; \label{prop:degeneratechar:downstring}
\end{enumerate}
\end{prop}
\begin{proof}
A string $C$ in $M$ is degenerate iff its dual string $C^\op$ in $M^\op$ is, so it suffices to prove $\eqref{prop:degeneratechar:deg} \iff \eqref{prop:degeneratechar:closedupstring} \iff \eqref{prop:degeneratechar:upstring}$.

The implication $\eqref{prop:degeneratechar:deg}\Rightarrow\eqref{prop:degeneratechar:closedupstring}$ follows immediately from Lemma~\ref{lemma:normalizedstrings}. Clearly, $\eqref{prop:degeneratechar:closedupstring}\Rightarrow\eqref{prop:degeneratechar:upstring}$.

For $\eqref{prop:degeneratechar:upstring}\Rightarrow\eqref{prop:degeneratechar:deg}$, let $C'\in \Str^+_{\sigma'}(a) \cap \Str_{m,n}$, $\sigma'=(m,n',I')>\sigma=(m,n,I)$ and $n'\leq n$. Denote the elements of $C$ by $a=a_m,a_{m+1},\dots,a_n=b$ and the elements of $C'$ by $a=a_m',a_{m+1}',\dots,a_{n}' = b'$.

According to the definition of the order, there exists a $k$ with $I \cap [m,k]=I'\cap [m,k]$ such that $k \in I-I'$. (The case $k=n$ is excluded because by assumption, $k \leq n'\leq n$ and thus in that case, $\sigma=\sigma'$.)

This means that $s(a_k) = a_{k+1}$, and either $k=n'<n$ or $k<n'$ and $t(a'_{k+1})=a'_k$.

In the former case, $C'$ must be strongly degenerate and also adheres to $\sigma$. However, $C-C'$ has bottom element $0$ and thus is also strongly degenerate. Thus $C=C'+(C-C')$ is degenerate.

In the latter case (i.e., $n'>k$ and $t(a'_{k+1})=a'_k$), we must have that $s(a'_k)=0$ and hence there is a degenerate string $C''=(a'_m,\dots,a'_k,0,\dots,0)$ adhering to the signature $\sigma$. Moreover, $C-C''$ is a degenerate string as well, thus $C$ is degenerate.
\end{proof}

\begin{corollary} \label{cor:maximalsignature}
Let $m < n \leq \infty$. A string $C \in \Str_{m,n} \cap \Str^+(a)$ is nondegenerate if and only if it has maximal signature among all strings in the same set. A string $C \in \Str^-(a) \cap \Str_{m,n}$ is nondegenerate if and only if it has minimal signature among all strings in the same set.
\end{corollary}
\begin{proof}
The second statement is exactly dual to the first, so we only prove the first.

For finite $n$, this immediately follows from Prop.~\ref{prop:degeneratechar} since the set of all signatures from $m$ to $n$ is finite.

For $n=\infty$, if $C$ is degenerate then by Lemma~\ref{lemma:degenerateinfinitestringchar}, some finite, upper closed substring $C|_m^{n'}$ is degenerate and thus admits an upper closed string of larger signature with the same endpoints. Extending that string with zeroes gives a string from $m$ to $\infty$ adhering to a larger signature.

Conversely, if $C$ is nondegenerate of signature $\sigma$ and there was another string in $\Str^+(a)$ with a larger signature $\sigma'$ then $\sigma$ and $\sigma'$ would differ at some finite stage $n$, and thus $C|_m^n$ would have smaller signature than $(C')_m^n$. Thus $C|_m^n$ would be degenerate, contradicting Lemma~\ref{lemma:degenerateinfinitestringchar}.
\end{proof}

\subsection{Height and order}

It follows from Prop.~\ref{prop:degeneratechar} that for every nonzero $a \in M_m$ and $n \geq m$, there exists at most one signature $\sigma(a,n)$ such that $\Str^+_{\sigma(a,n)}(a)$ contains a nondegenerate string. Moreover, since substrings of nondegenerate strings are nondegenerate, $\sigma(a,n+1)|_m^n = \sigma(a,n)$, which means that the sequence $\{\sigma(a,n)\}$ of all elements $\sigma(a,n)$ that exist, converges in the order topology to a unique signature.

\begin{defn}
The \emph{height} $\height(a)$ of a nonzero element $a \in M$ is the signature
\[
\lim_{n \to \infty} \sigma(a,n),
\]
where the limit is taken over the (finite or infinite) set of $n$ such that $\sigma(a,n)$ exists. 

Similarly, the \emph{order} $\order(a)$ of a nonzero element $a \in M$ is defined as dual of the height of its dual:
\[
\order(a) = (\height(a^\op))^\op
\]
where $a^\op \in M^\op$.

An element $a$ of height $\sigma$ does not necessarily admit a nondegenerate string in $\Str^+_\sigma(a)$; if it does, we call it a \emph{limit element}. (Of course, the issue only arises for elements of infinite height.) Similarly, if $a$ admits a nondegenerate string in $\Str^-_{\order(a)}(a)$, we call it a \emph{colimit element}. 

An $R$-module in which all elements are limit and colimit elements is called \emph{complete}.
\end{defn}

The difference between elements of infinite height and limit elements is analogous to the question of the vanishing of the derived inverse limit $\lim^1(\cdots \xrightarrow{p} A \xrightarrow{p} A)$ of the multiplication-by-$p$ map on an abelian $p$-group $A$. An element $a\in A$ can be $p^n$-divisible for all $n$ without admitting a sequence $(a_i)$ with $a_0=a$ and $pa_{i+1}=a_i$. The following example illustrates this idea for $R$-modules:

\begin{example}\label{ex:nonpureinjective}
For any two $R$-modules $M$, $N$ with $M_0= k\langle x\rangle$, $N_0= k\langle y\rangle$, $sx=0=sy$, we define the bouquet
\[
M \vee N = (M \oplus N)/(x-y).
\]
Let $\sigma=(0,\infty,I)$ be any signature with $0 \not \in I$ and such that $I$ has infinite complement. Note that $M(\sigma|_0^n)$ is a submodule of $M(\sigma)$ iff $n \not\in I$. Define an $R$-module as the infinite bouquet
\[
NPI(\sigma) = \bigvee_{n \not\in I} M(\sigma|_0^n).
\]

The module $NPI(0,\infty,\emptyset)$ is depicted below:
\[
\begin{tikzpicture}[>=stealth]
\graph[grow up, branch right,  nodes={math nodes,inner sep=1.5pt},empty nodes] { x_0 <- { x_{11}; x_{21} <- x_{22}; x_{31} <- x_{32} <- x_{33} }}; 
\node at (3,1) {$\cdots$};
\end{tikzpicture}
\]
We also define the module $LPI(\sigma)$ as
\[
LPI(\sigma) = NPI(\sigma) \vee M(\sigma).
\]
Denote the standard basis elements in the $M(\sigma)$ summand by $x_{\infty,i}$.

In both $NPI(\sigma)$ and $LPI(\sigma)$, $\height(x_0)=\sigma$, but only in $LPI(\sigma)$ is $x_0$ a limit element. By replacing the basis elements $x_{n,i}$ by $x_{n,i}-x_{\infty,i}$ in $LPI(\sigma)$, we see that
\[
LPI(\sigma) \cong M(\sigma) \oplus \bigoplus_{n \not\in I, n \geq 1} M(\sigma|_1^n)
\]
is isomorphic to a direct sum of string modules.

It will follow from Corollary~\ref{cor:lpicriterion} that $NPI(\sigma)$ cannot be a direct sum of string modules. We will continue the study of this example in Example~\ref{ex:nonpureinj:cont}.
\end{example}



\begin{lemma}\label{lemma:heightorderprops}$\;$
\begin{enumerate}
	\item If $\height(a)$ is finite then there exists an upper closed nondegenerate string of signature $\height(a)$ in $\Str^+(a)$. \label{lemma:heightorderprops:upperclosedexample}
	\item If $\height(a)=(m,n,I)$, $n \leq \infty$, then for every $m\leq n' \leq n$ with $n' \not \in I$, there exists an upper closed nondegenerate string of signature $\height(a)|_m^{n'}$ in $\Str^+(a)$.\label{lemma:heightorderprops:infiniteheightclosedstrings}
	\item If $\order(a)$ is finite then there exists a lower closed nondegenerate string of signature $\order(a)$ in $\Str^-(a)$. \label{lemma:heightorderprops:lowerclosedexample}
	\item If $C \in \Str^+(a)$ is upper closed then $\sigma(C) \leq \height(a)$. \label{lemma:heightorderprops:nondegclosedlowerbound}
	\item If $C \in \Str^-(a)$ is lower closed then $\sigma(C) \geq \order(a)$. \label{lemma:heightorderprops:nondegclosedupperbound}
\end{enumerate}
\end{lemma}
\begin{proof}
	\eqref{lemma:heightorderprops:nondegclosedlowerbound}: By Corollary~\ref{cor:maximalsignature}, there exists a nondegenerate upper closed string $C' \in \Str^+(a)$ with $\sigma(C')\geq\sigma(C)$, so we may assume $C$ is nondegenerate.
	Suppose that $\sigma=\sigma(C)=(m,n,I)$ and $\height(a)=(m,n',I')<\sigma$. Then we must have $n' > n$, $\sigma = \height(a)|_m^n$, and $n \in I'$. Let $C'$ be a nondegenerate string of signature $\sigma' = \height(a)_m^{n+1}$. Since $C$ is upper closed, $C \in \Str(\sigma')$ and is degenerate as such because its top element is zero. Similarly, $C'-C$ is degenerate because its bottom element is zero. Thus the sum, $C'$, is degenerate, a contradiction.

	\eqref{lemma:heightorderprops:upperclosedexample}: Let $C \in \Str_\sigma(a)$ be a nondegenerate string of signature $\sigma=(m,n,I)=\height(a)$ and suppose that $C$ is not upper closed. Let $C'$ be the unique extension of $C$ to $\Str_{m,n+1}$. This string $C'$ must be degenerate, because otherwise $\height(a)$ would not have maximal length. $C'$ adheres (uniquely) to the signature $\sigma'=(m,n+1,I \cup \{n\})<\sigma$. 	
	By Lemma~\ref{lemma:normalizedstrings}, there exists a normalized, upper closed string $C''\in \Str_{\sigma'}(a)$. Let $D \in \Sig_{m,\leq n}$ denote its nondegenerate part, which is also upper closed. By Lemma~\ref{lemma:normalizedstrings} again, $\sigma(D) > \sigma'$. Since in $\Sig_{m,\leq n+1}$, $\sigma$ is a direct successor to $\sigma'$, $\sigma(D) \geq \sigma$. But by \eqref{lemma:heightorderprops:nondegclosedlowerbound}, we also have $\sigma(D) \leq \sigma$, thus $D$ is a nondegenerate, upper closed string in $\Str_{\height(a)}(a)$.

	For \eqref{lemma:heightorderprops:infiniteheightclosedstrings}, let $I'=[m,\infty]-I$. For every $n \in I'$, let $C_n \in \Str_{\sigma|_m^{n+1}}(a)$ be a nondegenerate string. By assumption, $C_n|_m^n$ is upper closed and, as a substring of a nondegenerate string, nondegenerate.

	Assertions \eqref{lemma:heightorderprops:nondegclosedupperbound} and \eqref{lemma:heightorderprops:lowerclosedexample} are dual to \eqref{lemma:heightorderprops:nondegclosedlowerbound} and \eqref{lemma:heightorderprops:upperclosedexample}, respectively.
\end{proof}

\begin{lemma}\label{lemma:homomorphicimageofnondegenerateupperclosedstring}
Let $f\colon N \to M$ be a morphism of $R$-modules and $a \in N$. If $C \in \Str_\sigma^+(a)$ is a nondegenerate, upper closed string then $\sigma \leq \height_M(f(a))$.
\end{lemma}
\begin{proof}
If $f(C)$ is still nondegenerate then the claim follows from Lemma~\ref{lemma:heightorderprops}\eqref{lemma:heightorderprops:nondegclosedlowerbound}. 

If, on the other hand, $f(C)$ is degenerate in $M$ then by Prop.~\ref{prop:degeneratechar} there is a string $D$ in $\Str^+(f(a))$ adhering to a bigger signature $\tau$ , where $\tau$ and $\sigma$ have the same length. Hence, by Lemma~\ref{lemma:normalizedstrings}, there also exists a normalized string $C'\in \Str_\tau^+(f(a))$ with upper closed, nondegenerate part of signature $\tau'>\tau>\sigma$. Since $\tau' \leq \height(f(a))$ by Lemma~\ref{lemma:heightorderprops}, the claim follows.
\end{proof}

\begin{lemma}\label{lemma:height-basicprops}$\;$
\begin{enumerate}
	\item For $f\colon N \to M$ and $a \in N$, $\height_N(a) \leq \height_M(f(a))$ and $\order_N(a) \geq \order_M(f(a))$.
	\item Let $a \in N_m$, $b \in M_m$. Then
	\[
	\height_{N \oplus M}(a,b) = \min(\height_N(a),\height_M(b))
	\] 
	and
	\[
	\order_{N \oplus M}(a,b) = \max(\order_N(a),\order_M(b))
	\]
	\item If $a,b \in M_m$ with $a$, $b$, and $a+b$ nonzero, then
\[
\height(a+b) \geq \min(\height(a),\height(b)).
\]
and
\[
\order(a+b) \leq \max(\order(a),\order(b))
\]
\end{enumerate}
\end{lemma}
\begin{proof}$\;$
\begin{enumerate}
	\item
	If $\height(a)$ is finite then the claim follows from Lemma~\ref{lemma:homomorphicimageofnondegenerateupperclosedstring} and Lemma~\ref{lemma:heightorderprops}\eqref{lemma:heightorderprops:upperclosedexample}.
	If $\height(a)=\sigma=(m,\infty,I)$ is infinite then we consider two cases. If $[m,\infty]-I$ is finite, with maximal element $n$, then there exists a string $C \in \Str_\sigma^+(a)$ and Lemma~\ref{lemma:homomorphicimageofnondegenerateupperclosedstring} yields the result.

	In the remaining case, there exists a sequence of upper closed, finite, nondegenerate strings $C_i$ in $\Str_{\sigma_i}^+(a)$ with $\lim_{i \to \infty} \sigma_i = \height(a)$. Again by Lemma~\ref{lemma:homomorphicimageofnondegenerateupperclosedstring}, $\sigma_i \leq \height(f(a))$ for all $i$ and thus $\height(a)\leq\height(f(a))$.

	The claim about orders is dual.
	
	\item The projection $N \oplus M \to N$ shows that $\height_{N \oplus M}((a,b)) \leq \height_N(a)$, and similarly $\leq \height_M(b)$, so $\height_{N \oplus M}(a,b) \leq \min(\height_N(a),\height_M(b))$. To show equality, suppose that $\sigma=\height_N(a)\leq \height_M(b)=\tau$. If $a$ is a limit element, there exists a nondegenerate, upper closed string $C \in \Str_\sigma^+(a)$. By Lemma~\ref{lemma:filtering}, there also exists a (degenerate unless $\sigma=\tau$) upper closed string $D \in \Str_\sigma^+(b)$. Then the string $(C,D)\in \Str_{\sigma}^+((a,b))$ is upper closed and nondegenerate since its projection to $N$ is nondegenerate. By Lemma~\ref{lemma:heightorderprops}, $\height_{M\oplus N}((a,b)) \geq \sigma$. In case $a$ is not a limit element, choose a sequence $C_i$ of upper closed, nondegenerate strings with signatures converging to $\sigma$ and proceed as before.

	The claim about orders is dual.

	\item This follows from the previous part, using the addition homomorphism $M \oplus M \to M$.\qedhere
\end{enumerate}
\end{proof}

\subsection{\texorpdfstring{Injective $R$-modules}{Injective R-module}}

From now on, we will concentrate on nonnegatively graded $R$-modules. The injective objects are very easily classified:

\begin{lemma}
A nonnegatively graded $R$-module $M \in \Mod_R^+$ is injective in $\Mod_R^+$ if and only if the following three conditions are met:
\begin{enumerate}\renewcommand{\theenumi}{\roman{enumi}}
	\item $\ker s = \im t$ in all degrees;
	\item $\ker t = \im s$ in all positive degrees; and
	\item If $x \in \im(s)_{n+1}$ and $y \in \im(t)_{n-1}$ then there exists $z \in M_n$ such that $x=s(z)$ and $y = t(z)$. \label{lifting:typeiii}
\end{enumerate}
\end{lemma}
\begin{proof}
The three conditions boil down to the lifting problems for the monomorphisms which are the inclusions of the three ideals $(s)$, $(t)$, and $(s,t)$ into $R$, respectively. More precisely, since we are dealing with nonnegatively graded modules, they correspond to the inclusions
\[
(\Sigma^n I)_{\geq 0} \hookrightarrow (\Sigma^n R)_{\geq 0}
\]
for the three ideals $I=(s)$, $(t)$, and $(s,t)$, respectively, and any $n \geq 0$. This shows that injective $R$-modules satisfy the three conditions. Conversely, since the mentioned three ideals are precisely the graded prime ideals of $R$, the refined Baer criterion shows that the three conditions are sufficient for the injectivity of $M$.
\end{proof}

\begin{corollary}
A string module $M(\sigma) \in \Mod_R^+$ is injective if and only if $\sigma=(0,\infty,[0,n])$ for some $n \leq \infty$.
\end{corollary}
\begin{proof}
If $M$ was finite with top degree $n$ then $s(M_n)=0$ and hence $M_n \subseteq \im(t)$, a contradiction. So $M$ has to be infinite.

If $M$ had bottom degree $n>0$ then $t(M_n)=0$ and hence $M_n \subseteq \im(s)$, again a contradiction. So $M_0 \cong k$.

If there was some $n \in I$ with $n-1 \not\in I$ then we have a class $x_n \in M_n$, unique up to scaling, with $s(x_n) = x_{n+1} \neq 0$ and $t(x_n) = x_{n-1} \neq 0$. This means that the lifting problem of type  \eqref{lifting:typeiii} with $x=x_{n+1}$ and $y=0$ does not have a solution.
\end{proof}

\begin{corollary}
Any injective module in $\Mod_R^+$ is a direct sum of injective string modules.
\end{corollary}
\begin{proof}
First note that a nontrivial injective module $M$ must have nontrivial $M_0$. Thus let $M$ be injective and $a \in M_0 - \{0\}$. By injectivity, either $s^n(a) \neq 0$ for all $n$ or there is an $n \geq 0$ such that $s^n(a) \neq 0$ and there are classes $a_{n+1}, a_{n+2},\dots$ such that $t(a_{i+1})=a_i$ and $a_{n+1}=s^n(a)$. Thus $M$ contains an injective string module, which by injectivity must be a direct summand. By the Bass--Papp theorem and since $R$ is Noetherian, arbitrary direct sums of injectives over $R$ are injective, and thus the claim follows by transfinite induction.
\end{proof}

\subsection{Purity} \label{sec:purity}

For any ring $R$, a submodule $P$ of an $R$-module $M$ is called \emph{pure} if its inclusion map stays injective after tensoring with any other $R$-module. Direct summands are examples of pure submodules, and in fact every pure submodule is a filtered colimit of direct summands \cite[2.30]{adamek-rosicky:presentable-accessible}. The definition immediately generalizes to $\Z$- and $\N_0$-graded modules over graded rings.

The notion of purity is really only interesting when studying infinitely generated modules, and in fact turns out to be a central notion in that case. An equivalent characterization of purity is: $P<M$ is pure iff each (homogeneous) system of linear equations $Ax=\beta$ with $\beta \in P^m$, $A \in R^{m \times n}$ that is solvable in $M$, is also solvable in $P$ (cf. \cite[Theorem~4.89]{lam:modules-and-rings}).  

A module $I$ is called \emph{pure-injective} if $\Hom(M,I) \to \Hom(P,I)$ is surjective for every pure morphism $P \to M$, i.e. if maps into $I$ extend along pure monomorphisms. The module $I$ is \emph{algebraically compact} if every system of (infinitely many, homogeneous) $R$-linear equations
\begin{equation}\label{eq:biglinearsystem}
\Bigl\{\sum_{j \in J} r_{ij} x_j = b_i\Bigr\}_{i \in I}
\end{equation}
in possibly infinitely many variables $x_j$ has a solution under the assumption that every finite subset of equations does. Of course, every single equation in \eqref{eq:biglinearsystem} has only finitely many terms, i.e. for each $i$, $r_{ij}=0$ for all but finitely many $j$, but the required solution $x_j$ does not need to satisfy that all but finitely many $x_j$ are zero. Again, it is a standard result \cite{warfield:purity} that pure-injective modules are algebraically compact and vice versa, and this is true in the graded setting, as well.

In this subsection, we will show that linear systems of equations, as used in characterizing pure morphisms, can be replaced by the simpler notion of strings in our setting. We continue to only consider nonnegatively graded $R$-modules.

\begin{remark}
Note that the set of strings $C \in \Str_\sigma^+(a)$ of signature $\sigma=(m,n,I)$ (where $m \leq n \leq \infty$) is equal to the solution set of the system of linear equations in the variables $(a_i)_{i \in [m,n]}$:
\[
  a_m=a; \; \quad \begin{cases} a_m - t a_{m+1} = 0; & m \not \in I\\
  s a_m - a_{m+1} = 0 & m \in I \end{cases}
\]
\end{remark}

\begin{defn}
A submodule $P<M \in \Mod_R^+$ is \emph{pre-pure} if it is pure as a nonnegatively graded $k[s]$- and $k[t]$-module.

Concretely, this means that if $t^b x = \beta$ has a solution $x \in M$ (with $b \geq 0$, $\beta \in P$) then it also has a solution $x \in P$, and similarly for the equation $s^a x = \beta$.
\end{defn}

\begin{lemma}\label{lemma:puritycharacterization}
The following are equivalent for a submodule $P<M \in \Mod_R^+$:
\begin{enumerate}
	\item $P$ is pre-pure and for each $\beta \in P$ and nonnegative integers $n$, $a$, and $b$, the system of equations
\[
\begin{cases}
s^a x = 0\\
t^b x = \beta
\end{cases}
\]
is solvable in $M$ if and only if it is solvable in $P$. \label{lemma:puritycharacterization:atomicsystems}
	\item $P$ is pure. \label{lemma:puritycharacterization:pure}
	\item $P$ is pre-pure and every nondegenerate string in $P$ is also nondegenerate in $M$. \label{lemma:puritycharacterization:preservesnondegeneracy}
	\item For every $a \in P$, $\height_P(a)=\height_M(a)$ and $\order_P(a)=\order_M(a)$. \label{lemma:puritycharacterization:preservesheightandorder}
\end{enumerate}
\end{lemma}
\begin{proof}
$\eqref{lemma:puritycharacterization:atomicsystems}\Rightarrow\eqref{lemma:puritycharacterization:pure}$:
We first consider the special case of a system of equations of the form
\[
\begin{cases}
s^a x = \beta_1\\
t^b x = \beta_2.
\end{cases}
\]
for $a,b>0$. Suppose that there exists a solution $z \in M$. Then by pre-purity, there exists $y \in P$ such that $s^a y = \beta_1$. Now consider the system of equations
\[
\begin{cases}
s^a x = 0\\
t^b x = \beta_2-t^b y.
\end{cases}
\]
Since $z-y$ solves this system in $M$, there must be a solution $w$ in $P$ as well. But then $y+w$ solves the original system in $P$.

Now let $Ax=\beta$ be an arbitrary linear equation in $M$ with $\beta \in P^m$, $A \in R^{m \times n}$. Let $\alpha=(\alpha_1,\dots,\alpha_m)$ be the first column of $A$. We consider two cases:
\begin{itemize}
	\item All coefficients are multiples of some coefficient $\alpha_i$ for some $i$, without loss of generality $i=1$. This happens if either $\alpha_1 \in k^\times$, all $\alpha_i \in (s)$, or all $\alpha_i \in (t)$.	Then by a Gauss elimination step, $Ax=\beta$ is equivalent to a system of the form
	\[
	\left(
	\begin{array}{c|c}
	\alpha_1 & l\\
	\hline
	0 & B
	\end{array}
	\right)  x = \beta.
	\]
	By induction, $B x_{\geq 2} = \beta_{\geq 2}$ has a solution in $P$, and the first equation
	\[
	\alpha_1 x_1 = \beta_1-l\cdot x_{\geq 2}
	\]
	is solvable with $x_1 \in P$ by hypothesis because $\alpha_1 = \lambda s^a$ or $\alpha_1=\lambda t^b$ for some $\lambda \in k^\times$. (It could also happen that $\alpha_1=0$, but in that case the above equation would not be solvable in $M$ unless $\beta_1 - l \cdot x_{\geq 2} = 0$.)
	\item There are two indices $i$ and $j$ such that $\alpha_i=\lambda s^a$ and $\alpha_j = \mu t^b$, $a,b>0$, $\lambda, \; \mu \in k^\times$, and such that every $\alpha_q$ is a multiple of either $\alpha_i$ or $\alpha_j$. We assume without loss of generality that $(i,j)=(1,2)$. Then by a Gauss elimination step, $Ax=\beta$ is equivalent to a system of the form
	\[
	\left(
	\begin{array}{c|c}
	s^a & l_1\\
	\hline
	t^b & l_2\\
	\hline
	0 & B
	\end{array}
	\right)  x = \beta.
	\]
	Once again, by induction, $Bx_{\geq 2} = \beta_{\geq 3}$ has a solution in $P$, and the remaining equations are
	\[
	\begin{cases}
	s^a x_1 = \beta_1 - l_1 \cdot x_{\geq 2}\\
	t^b x_1 = \beta_2 - l_2 \cdot x_{\geq 2}
	\end{cases}
	\]
	By the special case discussed before, this system has a solution $x_1 \in P$.
\end{itemize}

$\eqref{lemma:puritycharacterization:pure}\Rightarrow\eqref{lemma:puritycharacterization:preservesnondegeneracy}$:
Suppose that $P$ is a pure submodule, so in particular it is pre-pure. To show that every nondegenerate string $C$ in $P$ stays nondegenerate in $M$, consider first the case that the inclusion $i\colon P \to M$ has a retraction $p\colon M \to P$, i.~e. that $P$ is a direct summand. Then $C = p(i(C))$, i.e. $i(C)$ is the preimage of a nondegenerate string and hence nondegenerate.

Now let $M_i$ be a filtered diagram of modules with colimit $M$ and $\iota_i\colon P \to M_i$ compatible split inclusions (the splittings $p_i\colon M_i \to P$ do not have to commute with the maps in the diagram.) Suppose $C = C_1 + \dots + C_n$ becomes degenerate in $M$, with each $C_j$ strongly degenerate. By Lemma~\ref{lemma:degenerateinfinitestringchar}, we may without loss of generality assume that $C$ is finite. Then $C_1,\dots,C_n$ lie in a common module $M_i$ and are strongly degenerate there, thus $C$ is degenerate in $M_i$, a contradiction to the first case.

$\eqref{lemma:puritycharacterization:preservesnondegeneracy}\Rightarrow\eqref{lemma:puritycharacterization:preservesheightandorder}$: Let $P$ be pre-pure, $a \in P$ and
\[
\sigma=(m,n,I)=\height_P(a)\neq\height_M(a)=(m,n',I')=\tau.
\]
By Lemma~\ref{lemma:height-basicprops}, we must have $\sigma<\tau$.
By the definition of the order relation, there exists $k$ such that $I \cap [m,k]=I' \cap [m,k]$, $k \not\in I'$, and either $k=n$ or $k \in I$. The former case is precluded by pre-purity. The latter case would imply by Prop.~\ref{prop:degeneratechar} that a nondegenerate string of signature $\sigma$ in $P$ would become degenerate in $M$, contrary to the assumption. The argument about the order is dual.

$\eqref{lemma:puritycharacterization:preservesheightandorder} \Rightarrow \eqref{lemma:puritycharacterization:atomicsystems}$:
Suppose that $P<M$ is a submodule with the property that $\height_P(a)=\height_M(a)$ and $\order_P(a)=\order_M(a)$ for all $a \in P$.

We consider the three cases of systems of equations:
\begin{enumerate}
	\item $t^b x = \beta$ for some $b \geq 1$, $\beta \in P_m$ has a solution $x \in M$. If the equation had no solution in $P$ then $\height_P(\beta)|_m^{m+b}<(m,n,\emptyset) = \height_M(\beta)|_m^{m+b}$ and thus $\height_P(\beta)<\height_M(\beta)$, a contradiction to the height condition.
	\item $s^a x = \beta$ for some $a \geq 1$, $\beta \in P_m$ has a solution $x \in M$. By the dual argument, a solution in $P$ must exist due to the order condition.
	\item $\begin{cases}s^a x = 0\\t^b x = \beta\end{cases}$ for some $a,b \geq 0$, $\beta \in P_m-\{0\}$ has a solution $x \in M$. We may without loss of generality assume that $s^{a-1}(x) \neq 0$ for all such solutions $x$ since we could otherwise decrease $a$.

	Thus there exists an upper closed string $C$ in $\Str^+(\beta)$ in $M$ adhering to the signature $\sigma=(m,m+a+b,[m+a,m+b-1])$, which by Lemma~\ref{lemma:normalizedstrings} we may assume to be nondegenerate. By Lemma~\ref{lemma:heightorderprops}, $\sigma \leq \height_M(\beta)=\height_P(\beta)$. If $\sigma = \height_P(\beta)$ then we are done, since a string of signature $\sigma$ in $P$ gives a solution in $P$ to the linear system. If $\sigma < \height_P(\beta)$ then the definition of the order implies that either there exists a solution $x \in P$ with $s^{a'}(x)=0$ for $a'<a$ (contradicting minimality of $a$) or there exists a longer nondegenerate string $C' \in \Str^+(\beta)$ with larger signature in $P$, which means that $C'|_m^{m+a+b}$ still solves the system. \qedhere
\end{enumerate}

\end{proof}

\begin{lemma}\label{lemma:stringpureinjective}
String modules are pure-injective.
\end{lemma}
\begin{proof}
Let $(b_i)$ be a string of signature $\sigma=(m,n,I)$ and $M$ the associated string module.

Consider a linear system of equations in a string module $M$ in variables $x_{ij}=\alpha_{ij}b_i$ in degree $i$, $j \in J_i$ some index sets $J_i$. By possibly increasing the number of variables, we can suppose that all equations are of one of the forms
\[
\sum_{j} \lambda_j \alpha_{ij} + \sum_k \mu_k \alpha_{i-1,j} + \sum_l \nu_l \alpha_{i+1,j} = \kappa,
\]
with the second summand only present if $i-1 \in I$ and the third only if $i \not\in I$.

This is a system of $k$-linear equations, and since $k$-vector spaces are pure-injective, any finitely solvable such system is solvable. Hence $M$ is pure-injective.
\end{proof}

\subsection{Local pure injectivity}\label{sec:local_pure_injectivity}

\begin{lemma}\label{lemma:piimplieslimit}
In a pure-injective $R$-module $M$, every nonzero element is a limit element.
\end{lemma}
\begin{proof}
Let $x_n \in M_n$ be a nonzero element of signature $\sigma$. The signature determines a countable system of equations of the form $tx_{i}=x_{i-1}$ or $sx_i=x_{i+1}$, of which each finite subsystem is solvable. By pure injectivity, the whole system is solvable, and the solution provides a nondegenerate string of signature $\sigma$.
\end{proof}

The converse of this statement is not true, since the property of every element being a limit element is closed under direct sums, but direct sums of pure-injectives are not typically pure-injective over $R$. For instance, the direct sum $M=\bigoplus_{i=1}^\infty M(0,i,\emptyset)$ is not pure-injective. If it were, then the inclusion $M \to M^\Pi=M=\prod_{i=1}^\infty M(0,i,\emptyset)$, as a pure map out of a pure-injective, would split. But in the quotient $M^\Pi/M$, every element is infinitely $t$-divisible, while in $M^\Pi$, no nontrivial element is. In terms of equations, if $a_i$ denotes the generator of $M(0,i,\emptyset)$ in degree $i$, the system
\[
\{ z - t^n x_n = a_1 + t a_2 + \cdots + t^n x_n\}_{n \geq 1}
\]
in the variables $z,x_1,x_2,\dots$ has no solution in $M$, but every finite subsystem does.

\begin{defn}
An $R$-module $M$ is called \emph{locally pure-injective} if every element $x \in M-\{0\}$ is a limit element.
\end{defn}

We will later show (Corollary~\ref{cor:lpicriterion}) that this definition of local pure injectivity agrees with the usual one.
Clearly, the class of local pure-injectives contains the pure-injectives and is closed under direct sums. 

\subsection{Balanced elements}

We record a couple of auxiliary lemmas:
\begin{lemma}\label{lemma:heightofelementsinheightstrings}
Let $C \in \Str^+_{\height(a)}(a)$ for some element $a\in M_n-\{0\}$, and $b \in C_m$ for some $m \geq n$. Then $\height(b) = \height(a)|_m^\infty$. 

Let $C \in \Str^-_{\order(a)}(a)$ for some element $a\in M_n-\{0\}$, and $b \in C_m$ for some $m \leq n$. Then $\order(b) = \order(a)|_0^m$. 
\end{lemma}
\begin{proof}
The nondegenerate string $C|_m^\infty$ shows that the height of $b$ is at least $\sigma = \height(a)|_m^\infty$. Suppose the $\height(b)>\sigma$. Then there exists an upper closed, nondegenerate string $D \in \Str^+(b)$ of signature greater than $\sigma$. Let $\tau = C|_a^b \cup D$ be the string obtained by splicing. Then the signature of $\tau$ is strictly larger than $\height(a)$, contradicting Lemma~\ref{lemma:heightorderprops}\eqref{lemma:heightorderprops:nondegclosedlowerbound}.

The second statement is dual.
\end{proof}

\begin{lemma}\label{lemma:splicedheightorder}
Let $a \in M_n$ be a limit element, $C^+ \in \Str^+_{\height(a)}(a)$, $C^- \in \Str^-_{\order(a)}(a)$ be nondegenerate, and $C = C^+ \cup C^-$ the spliced string. Then the inclusion $M(C) \to M$ is pure if and only if $C$ is nondegenerate.
\end{lemma}
\begin{proof}
By Lemma~\ref{lemma:puritycharacterization}\eqref{lemma:puritycharacterization:preservesnondegeneracy}, it suffices to show that if $M(C)$ is pre-pure and $C$ is nondegenerate then $M(C)$ is pure in $M$. Thus for $b \in C \cap M_m$, we need to consider four cases:
\begin{itemize}
	\item $m \geq n$ and $b$ is $t^r$-divisible in $M$ for some $r>0$ . Then $b$ must also be $t^r$-divisible in $M(C)$ because otherwise, we could construct an (upper closed) string in $C^+(a)$ of larger signature than $\height(a)$, contradicting Lemma~\ref{lemma:heightorderprops}\eqref{lemma:heightorderprops:nondegclosedupperbound}.
	\item $m \leq n$ and $b$ is $s^r$-divisible in $M$. This is dual to the previous case, using Lemma~\ref{lemma:heightorderprops}\eqref{lemma:heightorderprops:nondegclosedlowerbound}.
	\item $m > n$ and $b$ is $s^r$-divisible in $M$, say $s^r c=b$. Then the string
	\[(c,sc,\dots,s^rc=b) \in \Str^-(b) \cap \Str_{m-r}^m
	\]
	has minimal possible signature in the given set. By Corollary~\ref{cor:maximalsignature}, since $C|_{m-r}^m \in \Str^-(b)\cap \Str_{m-r}^m$ is nondegenerate, it cannot have larger signature. Thus $b$ is $s^r$-divisible in $M(C)$.
	\item $m<n$ and $b$ is $t$-divisible. This is dual to the previous case.
\end{itemize}
\end{proof}

This result informs the following definition:

\begin{defn}
For any $M \in \Mod_R$, a limit element $a \in M$ is called a \emph{balanced element} if it is contained in a pure string submodule of $M$.
\end{defn}

For any nonzero limit element $a \in M_n$, call a string $C$ a \emph{height-order string} for $a$ if $C|_n^{\infty}$ is nondegenerate of signature $\height(a)$, and $C|_{-\infty}^n$ is nondegenerate of signature $\order(a)$. Lemma~\ref{lemma:splicedheightorder} therefore says that $a$ is balanced if and only if it admits some nondegenerate height-order string.

%

\begin{example}
Every element in a string module is balanced, but there are unbalanced elements already in sums of two string modules. For instance, in Example~\ref{ex:illustrativestringexample}, the elements $a,c,d$, and $b-c$ are balanced, while $b$ is not. 
\end{example}

Indeed, in this example, the order of $b$ is larger than the order of $b-c$, so it is a non-example of the following criterion:

\begin{lemma}\label{lemma:balancedbasecase}
Let $a \in M_n$ be a nonzero limit element which has empty order, empty height, or minimal order among all elements in $M_n$. Then $a$ is balanced. \end{lemma}
\begin{proof}
Let $C$ be a height-order string for $a$.
If $\order(a)$ is empty then $C=C^+$ is nondegenerate and we are done. Similarly, if $\height(a)$ is empty then $C=C^-$ is nondegenerate. Now consider the case where $\order(a)$ has length $k>0$. We claim that $C$ is nondegenerate in this case as well. If it were not, we could write $C=C^u+C^l$ as in Corollary~\ref{cor:degenerateupperlowerdecomp} with $a=a^u+a^l$, $a^u \in C^u$, $a^l \in C^l$, and $C^u \in \Str^+_{\sigma(C)}(0)$. Both summands $a^u$, $a^l$ have to be nontrivial because otherwise, $C_+$ or $C_-$ would be degenerate. But then $C^u|_{n-k}^n \in \Str^-(a^u)$ is a lower closed string of a signature $\sigma^-<\order(a)$, contradicting Lemma~\ref{lemma:heightorderprops}\eqref{lemma:heightorderprops:nondegclosedupperbound}. 
Thus $C$ is pure and $a$ therefore balanced.
\end{proof}

\begin{lemma}\label{lemma:balanceddecomposition}
Let $P$ be a locally pure-injective module in $\Mod_R^+$. Then every nonzero element $a \in P_n$ can be written as a sum $a=a_1 +\cdots+a_k$ of balanced elements in such a way that $\order(a)=\order(a_i)$ for some $i$ and $\height(a)=\height(a_j)$ for some $j$.
\end{lemma}
\begin{proof}
We use induction on both $n$ and the order of $a$ within its degree, using Lemma~\ref{lemma:balancedbasecase} as base case.

Let $a \in P_n$ and $C$ a height-order string. If $C$ is nondegenerate, there is nothing to show, so suppose $C$ is degenerate.

Write $C=C^u + C^l$ as a sum of strongly degenerate strings as in Corollary~\ref{cor:degenerateupperlowerdecomp} and decompose $a = a^u + a^l$ correspondingly. We have that $\order(a^u)<\order(a)=\order(a^l)$ and $\height(a^u)=\height(a)>\height(a^l)$. We can thus apply the inductive assumption on $a^u$ and write $a^u$ as a sum of balanced elements, one of which has equal height to $a$. On the other hand, let $b^l$ be the bottom element of $C^l$, in degree $m<n$. Since the order of $b^l$ is empty, $b^l$ is pure. Let $D_i \in \Str^+_{\height(b^l)}(b^l)$ be a string adhering to the height of $b^l$ with component $a'$ in degree $n$. Then also $a'$ is pure. Moreover, both $C|_m^n$ and $D|_m^n$ are nondegenerate strings in $\Str^+(b^l)$, so they have equal signature, and $\order(a^l)=\order(a')$. Lastly, $a^l-a'$ has smaller order than $a$, allowing us to invoke the inductive assumption once more.
\end{proof}

We conclude:
\begin{corollary}\label{cor:kulikov}
Every nontrivial, locally pure-injective $R$-module $M$ contains a string submodule as a direct summand. Every element of $M$ is contained in a pure submodule which is isomorphic to a finite direct sum of string submodules. \qed
\end{corollary}

\begin{corollary}\label{cor:lpicriterion}
An $R$-module $M$ is locally pure-injective if and only if every element is contained in a pure submodule of $M$. \qed
\end{corollary}

\subsection{Classification results for (locally) pure-injective modules}

The following is an $R$-module version of \L o\'s's lemma for abelian groups \cite{los:abelian-groups}:

\begin{thm}\label{thm:retractofproduct}
The pure-injective $R$-modules are exactly the retracts of products of string modules.
\end{thm}
\begin{proof}
By Lemma~\ref{lemma:stringpureinjective}, string modules are pure-injective, and pure injectivity is closed under products and retracts, so any retract of a product of string modules is pure-injective.

Conversely, let $P$ be a pure-injective $R$-module.

By Lemma~\ref{lemma:balanceddecomposition}, $P$ is generated by balanced elements. For each balanced $a$, let $M(a)$ be a pure string submodule containing it. Pure maps out of pure-injectives split, thus let $p_a\colon M \to M(a)$ be a splitting with $p_a(a)=a$ for each balanced $a$, and define an embedding
\[
P \hookrightarrow \prod_{\substack{a \in P\\a \text{ pure}}} M(a).
\]
To see that this map is pure, we have to show that it preserves heights and orders also of unbalanced elements. This is guaranteed by the order and height properties of Lemma~\ref{lemma:balanceddecomposition}.
\end{proof}

\begin{corollary}
Every $R$-module can be purely embedded into a product of string modules.
\end{corollary}
\begin{proof}
Every $R$-module $M$ embeds purely into its pure-injective hull $P$, which by Theorem~\ref{thm:retractofproduct} can be purely embedded into products of string modules.
\end{proof}

\begin{thm}\label{thm:pureinjclassification}
The pure-injective $R$-modules are exactly the pure-injective hulls of direct sums of string modules.
\end{thm}
\begin{proof}
Consider the set $S$ of all collections $\mathcal C=(C_i)$ of nondegenerate strings in $M$ with bottom element in any degree, finite or infinite, such that
\begin{itemize}
	\item Every $C_i$ spans a pure string submodule $M_i$ of $M$;
	\item The map $N_{\mathcal C} = \bigoplus_i M_i \to M$ is injective.
\end{itemize}
The set $S$ is ordered by inclusion and is nonempty since $M$ contains balanced elements by Lemma~\ref{lemma:balancedbasecase}. For every $\mathcal C \in S$, $N_{\mathcal C}<M$ is pure. Since directed colimits of pure submodules are pure, and directed colimits of direct sums by maps that are inclusions of summands are again direct sums, the ascending chain condition in Zorn's lemma is satisfied, so let $\mathcal C$ be a maximal element and $N=N_{\mathcal C}$ the associated module. Then $N \to M$ is a pure inclusion of a locally pure-injective submodule into $M$.

Let $\hat N$ be a pure-injective hull of $N$ in $M$, i.e. a maximal pure-essential extension of $N$ in $M$. We want to show that $\hat N=M$. By \cite[\textsection 3 Prop.~6]{warfield:purity}, $\hat N$ is a direct summand of $M$, and $M/\hat N$ is pure-injective. By Cor.~\ref{cor:kulikov}, there is a pure string submodule $M'$ of $M/\hat N \subset M$. Then by definition, $N \oplus M' \to M$ is injective, contradicting maximality.
\end{proof}

\begin{thm}\label{thm:countablelpiclassification}
Let $M$ be a countably generated, locally pure-injective $R$-module. Then $M$ is isomorphic to a countable direct sum of string modules.
\end{thm}
\begin{proof}
We prove inductively that given any sequence $(x_i)$ of elements of $M$, there exists a family $(M_i)$ of finite sums of string modules such that $\bigoplus_i M_i$ is pure in $M$ and contains all $x_i$.

Suppose that we have found such submodules $M_1,\dots,M_n$ such that $\bigoplus_{i=1}^n M_i$ is a direct summand of $M$ and contains $x_1,\dots,x_n$, and let $N$ be a complement. Write $x_{n+1}=y+z$ with $y \in \bigoplus_{i=1}^nM_i$ and $z \in N$. As a direct summand of $M$, $N$ is locally pure-injective. By Corollary~\ref{cor:kulikov}, $z$ is contained in a pure submodule $M_{n+1}<N<M$ which is isomorphic to a finite direct sum of string submodules. 

Passing to the limit, $\bigoplus_{n=1}^\infty M$ is pure in $M$ and contains all $x_i$.

The theorem follows by choosing $(x_i)$ to be a countable list of generators.
\end{proof}

We next address the question of uniqueness of the decomposition of Theorem~\ref{thm:pureinjclassification}. This is essentially the Krull-Schmidt-Azumaya theorem in a graded context. We use it in the form proved by Gabriel:

\begin{thm}[{\cite[\textsection I.1, Th\'eor\`eme 1]{gabriel:categories-abeliennes}}] \label{thm:gabriel-krull-schmidt}
Let $\mathcal C$ be a Grothendieck category. If $(M_i)_{i \in I}$ and $(N_j)_{j \in J}$ are two families of indecomposable objects in $\mathcal C$ with local endomorphism rings such that $\bigoplus_{i\in I} M_i \cong \bigoplus_{j \in J} N_j$ then there exists a bijection $h\colon I \to J$ such that $M_i \cong N_{h(i)}$. \qed
\end{thm}

\begin{corollary}\label{cor:uniquemodpsplitting}
Let $M \cong \bigoplus_{j=1}^{m} M(\sigma_j) \xrightarrow[\cong]{\phi} \bigoplus_{i=1}^{m'} M(\sigma'_i)$. Then the sets $\{\sigma_j\}_j$ and $\{\sigma'_i\}_i$ are equal.
\end{corollary}
\begin{proof}
A Grothendieck category is an AB5 category with a generator, and the category of nonnegatively graded $R$-modules satisfies these conditions. 
We also have that $\End(M(\sigma)) \cong k$ is indeed local. Thus Theorem~\ref{thm:gabriel-krull-schmidt} applies to give the result.
\end{proof}

\section{\texorpdfstring{The classification of abelian $p$-torsion Hopf algebras}{The classification of abelian p-torsion Hopf algebras}} \label{sec:p-torsion-hopf-alg}

With the results of the previous section in hand, all ingredients are in place to prove the first main theorem.

Continue to let $k$ be a perfect field of characteristic $p$. As an object of an abelian category, any Hopf algebra $H$ has a multiplication-by-$p$ map, classically denoted by $[p]\colon H \to H$. This map is given by the $p$-fold comultiplication followed by the $p$-fold multiplication:
\[
[p]\colon H \xrightarrow{\Delta^{p-1}} H^{\otimes p} \xrightarrow{\mu^{p-1}} H.
\]

In this section, we classify certain classes of $p$-torsion Hopf algebras. From an algebro--geometric point of view, $p$-torsion Hopf algebras represent $\F_p$-module schemes rather than abelian group schemes. Since the Dieudonné equivalence (Theorem~\ref{thm:dieudonneeq}) is an equivalence of abelian categories, the full subcategory of $p$-torsion Hopf algebras is equivalent to the category of $p$-torsion Dieudonné modules. By Lemma~\ref{lemma:splitting}, the category of $p$-typical $p$-torsion Hopf algebras (of any type) is thus equivalent to the category of nonnegatively graded modules over $R=\DRing/(p)\cong k[s,t]/(st)$.

\begin{proof}[Proof of Theorem~\ref{thm:modpclassification}]
We claim that the Hopf algebra $H(r,m,n,I) = D^{(r)} M(m,n,I)$, using the $p$-typical Dieudonné equivalence \eqref{eq:ptypicaldieudonne}, is uniquely determined by the characterization of $H(r,m,n,I)$ in the introduction. Let $H$ be any Hopf algebra isomorphic as algebras to 
\[
k[x_m,x_{m+1},\dots,x_n]/\Biggl(x_{i}^p - \begin{cases} x_{i+1}; &i \in I\\0;& i \geq n \text{ or } i \not\in I\end{cases}\Biggr), \quad |x_i| = rp^i.
\]
Then $H$ is $j$-typical, and $M=D^{(j)}(H) \cong \langle y_m,y_{m+1},\dots,y_n\rangle$ with $s(y_{i})=y_{i+1}$ iff $i \in I$. In order for $H$, and hence $M$, to be indecomposable, it is necessary that $ty_{i+1} = y_{i}$ iff $i \not\in I$, because otherwise, $M$ would split as $\langle y_m,\dots,y_{i}\rangle \oplus \langle y_{i+1},\dots,x_n\rangle$. Thus $M \cong M(m,n,I)$. 

By Lemma~\ref{lemma:splitting}, every Hopf algebra splits naturally and uniquely into a tensor product of $p$-typical parts of various types, so it suffices to show the theorem for $p$-typical Hopf algebras of type $j$. 
As the Dieudonné equivalence is one of abelian categories, $H$ is pure-injective if and only if $D^{(j)}H$ is, and the first part follows from Corollary~\ref{cor:kulikov}, the second part from Theorem~\ref{thm:retractofproduct}, and the third part from Theorem~\ref{thm:pureinjclassification}.
\end{proof}

\begin{lemma}\label{lemma:lpiequivalence}
A $p$-torsion Hopf algebra $H$ is locally pure-injective (in the sense of the definition given in the introduction) iff the Dieudonné modules $D^{(j)}(H)$ are locally pure-injective (in the sense of Subsection~\ref{sec:local_pure_injectivity}) for all $j>0$.
\end{lemma}
\begin{proof}
Since the $F$ and $V$ operators preserve the $p$-typical splitting, it is clear that $H$ is locally pure-injective iff all of its $p$-typical factors are locally pure-injective, thus assume that $H$ is locally pure-injective and $p$-typical of type $j$. Under the Dieudonné equivalence $D^{(j)}\colon \Hopf \to \Mod_{\mathcal R}$, we have that
\[
D^{(j)}(F) = s \quad \text{and} \quad D^{(j)}(V)=t.
\]
Moreover, $P(H)\cong \ker(s)\colon D^{(j)}(H) \to D^{(j)}(H)$, and hence every $x \in P(H)$ corresponds to an element $D^{(j)}(x) \in \ker(s)$. (This is a slight abuse of notation, since there is no natural transformation $H \to D^{(j)}(H)$.) Thus the condition on $x$ for $H$ being locally pure-injective is equivalent to $D^{(j)}(x)$ being a limit element.
\end{proof}

\begin{proof}[Proof of Theorem~\ref{thm:modpcountableclassification}]
Combine Lemma~\ref{lemma:lpiequivalence} and Theorem~\ref{thm:countablelpiclassification}, noting that a $p$-torsion Hopf algebra is countably generated if and only if its Dieudonné module is.
\end{proof}

\begin{example}
Primitively and countably generated Hopf algebras are Hopf algebras for which the canonical map $PH \to \tilde H \to QH$ from primitives to indecomposables is surjective. For the associated $\DRing$-modules $M$, this translates to the map $\ker(t) \to \coker(s)$ being surjective. Thus for each $m \in M_i$, there exists an $m' \in M_{i-1}$ such that $t(m+sm')=0$. By induction, $tm'=0$ and hence $tm=0$. This shows that primitively generated Hopf algebras have trivial Verschiebung, and in particular are $p$-torsion and locally pure-injective. Hence they split into copies of $H(r,m,I)$. The Hopf algebra $H(r,m,I)$ is primitively generated exactly if $I=[0,m-1]$, recovering the abelian case of the classification in \cite[Theorem~7.16]{milnor-moore:hopf} and generalizing it to countably generated Hopf algebras. This is not new -- by \cite{webb:graded-modules}, \emph{any} connected $\F_p[s]$-module splits into a sum of cyclic modules.
\end{example}

\section{Hopf algebras that are free as algebras or cofree as coalgebras}

A classification of general abelian Hopf algebras seems too much to ask for, even in the finite-type case; we will give some hopefully illuminating examples of the complexity of the problem in the next section. However, we can classify the Hopf algebras that reduce to the Hopf algebras $H(r,m,n,I)$ of Section~\ref{sec:p-torsion-hopf-alg} and use this to classify all Hopf algebras that are either free as algebras or cofree as coalgebras and whose mod-$p$ reduction is (locally) pure-injective.

\begin{remark}
The forgetful functor from the category $\Hopf$ to connected, graded, commutative algebras has a right adjoint, the cofree Hopf algebra on an algebra. Similarly, the forgetful functor from $\Hopf$ to connected, graded, cocommutative coalgebras has a left adjoint, the free Hopf algebra on a coalgebra. A Hopf algebra that is free on a coalgebra is also free as an algebra, but not vice versa. Similarly, a Hopf algebra that is cofree on an algebra is also cofree as a coalgebra, but not vice versa.
\end{remark}

\begin{defn}
A nontrivial Hopf algebra $H$ is called \emph{basic} if its mod-$p$ reduction $H/[p]$ is indecomposable pure-injective.
\end{defn}

By Theorem~\ref{thm:modpclassification}, $H$ is thus basic iff $H/[p] \cong H(r,\sigma)$ for some $r$ coprime to $p$ and some signature $\sigma=(m,n,I)$. Both $r$ and $\sigma$ are uniquely determined by $H$.

We will use the graphical model of ``basic Hopf graphs'' from the introduction for indexing basic Hopf algebras. Projection to the $x$-axis associates to each basic Hopf graph $\Gamma$ a pair $\sigma_\Gamma=(m_\Gamma,n_\Gamma,I_\Gamma)$, where $m$ is the minimal $x$-coordiante of all arrow end points, and $n$ is the maximal $x$-coordinate of all arrow end points (possibly $\infty$) and $i \in I$ iff there is an arrow $(i,j)$ to $(i+1,j')$ for some $j$, $j'$. Moreover, each Hopf graph $\Gamma$ comes with a vector $v_\Gamma=(v_{\Gamma,m},\dots,v_{\Gamma,n}) \in \N_0^{n_\Gamma-m_\Gamma+1}$ consisting of the $y$-coordinates of the vertices of $\Gamma$.

Conversely, a signature $\sigma=(m,n,I)$ and a vector $v \in \N_0^{n-m+1}$ gives rise to a unique basic Hopf graph under the conditions
\begin{enumerate}
	\item $v_m=0$
	\item if $i \in I$ then $v_{i+1} \in \{v_{i}-1,v_{i}\}$;
	\item if $i \not \in I$ then $v_{i+1} \in \{v_{i}, v_{i}+1\}$;
	\item if $n < \infty$ then $v_{n}=0$.
\end{enumerate}

Fig.~\ref{fig:hopfgraphone} and \ref{fig:hopfgraphtwo} show the two examples of basic Hopf graphs from the introduction together with data $m,I,v$ explained above.

\begin{figure}[ht]
\begin{tikzpicture}[scale=0.5]
\draw[gray,very thin] (0,0) grid (5,3);
\draw[stealth-] (0,0) -- (1,1);
\draw[-stealth] (1,1) -- (2,1);
\draw[stealth-] (2,1) -- (3,1);
\draw[-stealth] (3,1) -- (4,0);
\end{tikzpicture}
\caption{$m_\Gamma=0$, $n_\Gamma=4$, $I_\Gamma = \{1,3\}$, $v_\Gamma = (0,1,1,1,0)$}\label{fig:hopfgraphone}
\end{figure}

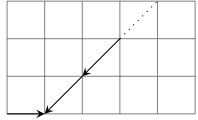
\begin{figure}[ht]
\begin{tikzpicture}[scale=0.5]
\draw[gray,very thin] (0,0) grid (5,3);
\draw[-stealth] (1,0) -- (2,0);
\draw[stealth-] (2,0) -- (3,1);
\draw[stealth-] (3,1) -- (4,2);
\draw[dotted] (4,2) -- (5,3);
\end{tikzpicture}
\caption{$m_\Gamma=1$, $n_\Gamma=\infty$, $I_\Gamma = \{1\}$, $v_\Gamma = (0,0,1,2,3,\dots)$}\label{fig:hopfgraphtwo}
\end{figure}

\begin{defn}
Let $\Gamma$ be a basic Hopf graph. Define an $\DRing$-module $M(\Gamma)$ as follows:
\[
M(\Gamma)_i = \begin{cases}
W(k)/p^{(v_\Gamma)_i+1}\langle x_i\rangle; & m_\Gamma \leq i \leq n_\Gamma\\
0; & \text{otherwise}
\end{cases}
\]
together with $s$- and $t$-multiplications given, for $1 \leq i \leq m$, by
\begin{equation}\label{eq:staction}
sx_{i} = \begin{cases}
x_{i+1}; & i \in I_\Gamma\\
px_{i+1}; & i \not \in I_\Gamma
\end{cases}; \quad 
tx_{i+1} = \begin{cases} 
px_{i}; &i \in I_\Gamma\\
x_{i}; &i \not \in I_\Gamma.
\end{cases}
\end{equation}
\end{defn}

We call an $\DRing$-module $M$ \emph{basic} if its associated Hopf algebra $(D^{(j)})^{-1}(M)$ is basic (for any, and hence all, $j$), or equivalently, by Theorem~\ref{thm:modpclassification}, if $M/(p)$ is a string module.

\begin{thm}\label{thm:classificationofbasicmodules}
An $\DRing$-module $M$ is basic with $M/(p) \cong M(\sigma)$ if and only if $M \cong M(\Gamma)$ for some basic Hopf graph $\Gamma$ with $\sigma=\sigma_\Gamma$. Moreover, $M$ is uniquely determined up to isomorphism by the properties
\begin{enumerate}
	\item $t\colon M_i \to M_{i-1}$ is injective iff $\Gamma$ contains an arrow from $(i,j)$ to $(i-1,j)$ or an arrow from $(i-1,j)$ to $(i,j-1)$, and
	\item $s\colon M_{i-1} \to M_i$ is injective iff $\Gamma$ contains an arrow from $(i-1,j)$ to $(i,j)$ or an arrow from $(i,j)$ to $(i-1,j-1)$.
\end{enumerate}
\end{thm}
\begin{proof}
It is evident from the definition that $M(\Gamma)/(p) \cong M(\sigma)$, so the ``only if'' direction is clear.

Let $M$ be a basic $\DRing$-module with $M/(p) \cong M(\sigma)$. Then $M_i=W_{v_i+1}(k)$ for some positive integers $v_i$ for $m \leq i\leq n$ and $M_i=0$ otherwise. Note that the condition $p=st=ts$ implies that $v_i \in \{v_{i-1}-1,v_{i-1},v_{i-1}+1\}$ for all $i \geq 0$.

We inductively construct lifts $\tilde x_i \in M_i$ of the defining generators $x_i \in M(\sigma)_i$ and show that the $v_i$ satisfy the conditions for being $v_\Gamma$ of a basic Hopf graph $\Gamma$ with $\sigma=\sigma_\Gamma$. Since $tM_m=0$ for dimensional reasons, $pM_m=0$ as well, so $M_m=k$ and $v_m=0$. Let $\tilde x_m$ be the unique lift of $x_m$.

Suppose now that $\tilde x_m,\dots,\tilde x_{i}$ are constructed satisfying \eqref{eq:staction}. If $i \in I$, define $\tilde x_{i+1} = s(\tilde x_{i})$. Then $\tilde x_{i+1}$ is a lift of $x_{i+1}$ and hence a generator, and we have that $t\tilde x_{i+1} = ts\tilde x_{i}=p\tilde x_{i}$, so \eqref{eq:staction} is satisfied for the index $i$. Since $s\colon M_{i} \to M_{i+1}$ is a surjection (it maps a generator to a generator), we cannot have $v_{i+1}=v_{i}+1$.

Conversely, if $i \not \in I$, but $i \leq m$, $t\colon M_{i+1}/(p) \to M_{i}/(p)$ is bijective, hence $t\colon M_{i+1} \to M_{i}$ is surjective. Let $\tilde x_{i+1}$ be any preimage of $\tilde x_i$ under $t$. Then $\tilde x_{i+1}$ is a lift of $x_{i+1}$ and hence a generator. Again, $v_{i+1}=v_{i}-1$ is not possible because of the surjectivity of $t\colon M_{i+1} \to M_{i}$.
\end{proof}

\begin{proof}[Proof of Theorem~\ref{thm:basichopfclassification}]
Given $\Gamma$, define $H(\Gamma) = (D^{(r)})^{-1}(M(\Gamma))$, using the $p$-typical Dieudonné functor of type $r$ from \eqref{eq:ptypicaldieudonne}. By construction, $H(\Gamma)/[p] \cong H(r,m_\Gamma,n_\Gamma,I_\Gamma)$. For any $p$-typical Hopf algebra $H$ of type $r$ with $M=D^{(r)}(H)$, we have that
\begin{align*}
D^{(r)}(F\colon H_{rp^i} \to H_{rp^{i+1}}) &= s\colon M_i \to M_{i+1}\\
\intertext{and}
D^{(r)}(V\colon H_{rp^{i+1}} \to H_{rp^i}) &= t\colon M_{i+1} \to M_{i}.
\end{align*}
The claimed classification of basic Hopf algebras thus follows directly from Theorem~\ref{thm:classificationofbasicmodules}. Moreover, let $H$ be any Hopf algebra whose mod-$p$ reduction is indecomposable. Let $r$ be the smallest positive integer such that $H_{rp^m}\neq 0$ for some $m \geq 0$. By Theorem~\ref{thm:modpclassification}, $H$ is concentrated in degrees $rp^j$. Then $M=D^{(r)}(H) \cong M(\Gamma)$ is basic, and $H \cong H(r,\Gamma)$.
\end{proof}

It seems that most Hopf algebras one encounters ``in nature'' are tensor products of basic ones. We will not attempt to support this claim with many examples, but one important one is the following:

\begin{example}\label{ex:gammap}
Consider the Hopf algebra $H=H^*(BU;k) \cong k[c_1,c_2,\dots]$, where $BU$ is the classifying space of the infinite unitary group and $c_i$ are the universal Chern classes in degree $2i$. By dimensional considerations, $H$ splits into a tensor product of one $p$-typical part of type $j$ for each $p \nmid j$. The type-1 part $\Lambda_p$ is the Hopf algebra representing the $p$-typical Witt vector functor. It is basic: $\Lambda_p \cong H(\Gamma)$ for 
\[
\Gamma\colon \begin{matrix}\begin{tikzpicture}[scale=0.5]
\draw[gray,very thin] (0,0) grid (4,4);
\draw[stealth-] (0,0) -- (1,1);
\draw[stealth-] (1,1) -- (2,2);
\draw[stealth-] (2,2) -- (3,3);
\draw[dotted] (3,3) -- (4,4);
\end{tikzpicture}
\end{matrix}
\]
\end{example}

\begin{lemma}
A Hopf algebra $H$ is free as a commutative algebra iff the Frobenius $F$ on $DH$ is injective. It is cofree as a cocommutative coalgebra iff the Verschiebung $V$ on $DH$ is surjective.
\end{lemma}
\begin{proof}
The $p$th power map of $H$ is a $\frob$-twisted $k$-linear Hopf algebra morphism $F\colon H \to H$ inducing the maps $F\colon DH \to DH$ of Dieudonné modules. If $H$ is free as a commutative algebra, $F$ will therefore be injective. Conversely, suppose $H$ is not free as a commutative algebra. Every Hopf algebra over a field is a directed colimit of Hopf algebras that are finitely generated as algebras. If $\{x_i\}$ is a minimal set of algebra generators for $H$, then there exists a polynomial $q$ in the $x_i$ which is zero in $H$. Since this polynomial only involves finitely many of the $x_i$, it is contained in a finitely generated Hopf algebra. Thus it suffices to show that $F$ is noninjective in the case where $H$ is of finite type. By Borel's theorem on the classification of algebra structures on connected, graded Hopf algebras \cite[Theorem~7.11]{milnor-moore:hopf}, a Hopf algebra of finite type that is not free as an algebra has a noninjective Frobenius. The argument for the Verschiebung is dual.
\end{proof}

This gives a simple characterization of when a basic Hopf graph corresponds to a free resp. cofree Hopf algebra:
\begin{corollary}
The Hopf algebra $H(\Gamma)$ is
\begin{itemize}
	\item free as an algebra if $\Gamma$ is infinite and consists of horizontal right arrows and diagonal left arrows only;
	\item cofree as a coalgebra if $\Gamma$ is infinite and consists of horizontal left arrows and diagonal left arrows only.
\end{itemize}
\end{corollary}
Thus the only basic Hopf algebra that is both free as an algebra and cofree as a coalgebra is $\Lambda_p$ of Example~\ref{ex:gammap}.

\begin{thm} \label{thm:freeorcofree}
Let $H$ be a countably generated Hopf algebra whose mod-$[p]$ reduction is locally pure-injective. If $H$ is free as an algebra or cofree as a coalgebra then $H$ is isomorphic to a tensor product of basic Hopf algebras.
\end{thm}
\begin{proof}
Without loss of generality, suppose that $H$ is $p$-typical and $M$ its associated $\DRing$-module.
Choose an isomorphism
\[
\bigoplus_{i=1}^n\bar\gamma_i\colon \bigoplus_{i=1}^n  M(\sigma_i)\xrightarrow{\cong} M/(p)
\]
as in Theorem~\ref{thm:modpcountableclassification}, where $n \in \N \cup \{\infty\}$. By Theorem~\ref{thm:basichopfclassification}, there are basic Hopf graphs $\Gamma_i$ and maps 
\[
\gamma_i\colon \Sigma^{r_i} M(\Gamma_i) \to M
\]
such that $\bigoplus_{i=1}^n \gamma_i$ is a surjective map and $\bar\gamma_i = \gamma_i/(p)$.

We denote the defining generator of $\Sigma^{r_i} M(\Gamma_i)$ in degree $j$ by $x_{i,j}$ and identify it with its image in $M$ under $\gamma_i$. To see that $\bigoplus_{i=1}^n \gamma_i$ is an injection, we need to show that
\begin{equation}\label{eq:linindep}
\sum_{i=1}^n \alpha_i x_{i,j} = 0 \Longrightarrow \alpha_i x_{i,j}=0 \text{ for all } \alpha_i \in k.
\end{equation}
Here we implicitly assume that only finitely many $\alpha_i$ are nontrivial if $n=\infty$.

We show this by induction on the degree $j$. For $j=0$, $M(\Gamma_i)_0 \cong M(\sigma_i)_0$ is $p$-torsion, and by definition, the nonzero $x_{i,0}$ are linearly independent. Assume that \eqref{eq:linindep} holds for $j-1$. Since modulo $p$, the $x_{i,j}$ are linearly independent, we have that $p \mid \alpha_i$ for all $i$. But $p=st$, and since $s$ is injective by assumption, we have that
\[
\sum_{i=1}^n \frac{\alpha_i}p tx_{i,j}=0.
\]
We have that $tx_{i,j} \in \{x_{i,j-1},px_{i,j-1}\}$, so by induction, $\frac{\alpha_i}p tx_{i,j}=0$ and hence $s(\frac{\alpha_i}ptx_{i,j})=\alpha_i x_{i,j}=0$ for all $i$.

This completes the proof that $\bigoplus_{i=1}^n \gamma_i$ is an isomorphism.

The result for Hopf algebras which are cofree as coalgebras follows from dualization.
\end{proof}

\begin{proof}[Proof of Theorem~\ref{thm:freeorcofreeconcrete}]
Given Theorem~\ref{thm:freeorcofree}, the only thing that remains is to classify those basic Hopf algebras of mod-$p$ type $\sigma$ which are free as algebras (the other case being dual). This corresponds to $s$ being injective. The characterization of basic $\DRing$-modules with injective $s$ from Theorem~\ref{thm:classificationofbasicmodules} shows that there is exactly one basic Hopf graph $\Gamma$ with this property for any given $\sigma$.
\end{proof}

We will now compare the property of $H$ being free as an algebra to two stronger conditions. The first condition is being a projective object. The second condition is, morally, to be free over a graded, connected coalgebra. However, this notion is unnecessarily restrictive and is incompatible with the property of being $p$-typical -- a $p$-typical Hopf algebra can never be free in this sense. Instead, we consider the property of being free over a graded \emph{$p$-polar} coalgebra \cite{bauer:p-polar,bauer:p-polar-iterated-loops}. 

\begin{defn}
A \emph{graded $p$-polar $k$-coalgebra} is a graded vector space $C$ together with a $k$-linear map
\[
\Delta\colon C_{pn} \to \Bigl(C_n \otimes \cdots \otimes C_n\Bigr)^{\Sigma_p}
\]
which, with this structure, is a retract of a graded $k$-coalgebra.
\end{defn}

The dual definition of this notion was given in \cite{bauer:p-polar-iterated-loops}, along with the dual version of the following theorem:

\begin{thm}
The functor which associates to a graded $k$-coalgebra the free graded Hopf algebra over it factors naturally through the category of $p$-polar coalgebras.
\end{thm}
In other words, there exists a free Hopf algebra functor on $p$-polar coalgebras, and the usual free Hopf algebra functor on coalgebras really only depends on the underlying $p$-polar coalgebra structure. This free functor sends $p$-typical $p$-polar coalgebras (in the obvious meaning) to $p$-typical Hopf algebras of the same type.

\begin{prop}
Let $H=H(r,\Gamma)$ be an irreducible Hopf algebra. Then
\begin{enumerate}
	\item $H$ is free on a graded, connected $p$-polar coalgebra iff $\Gamma$ has $n_\Gamma= \infty$, $I_\Gamma = \{ i \mid i \geq k\}$, $(v_H)_i = \min(i-m_\Gamma,n)$ for $m_\Gamma \leq i \leq \infty$ (Figs.~\ref{fig:cofreeonfinite} and \ref{fig:cofreeoninfinite})

\begin{figure}[ht]
\begin{tikzpicture}[scale=0.5]
\draw[gray,very thin] (0,0) grid (5,3);
\draw[stealth-] (0,0) -- (1,1);
\draw[stealth-] (1,1) -- (2,2);
\draw[stealth-] (2,2) -- (3,3);
\draw[-stealth] (3,3) -- (4,3);
\draw[dotted] (4,3) -- (5,3);
\end{tikzpicture}
\caption{$m=0$, $n=\infty$, $I = [3,\infty]$, $v_H = (0,1,2,3,3,3,\dots)$}\label{fig:cofreeonfinite}
\end{figure}

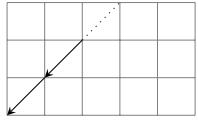
\begin{figure}[ht]
\begin{tikzpicture}[scale=0.5]
\draw[gray,very thin] (0,0) grid (5,3);
\draw[stealth-] (0,0) -- (1,1);
\draw[stealth-] (1,1) -- (2,2);
\draw[dotted] (2,2) -- (3,3);
\end{tikzpicture}
\caption{$m=0$, $n=\infty$, $I = \emptyset$, $v_H = (0,1,2,3\dots)$}\label{fig:cofreeoninfinite}
\end{figure}
\item $H$ is projective if, in addition, $m=0$.
\end{enumerate}
\end{prop} 
\begin{proof}
\begin{enumerate}
	\item In \cite{kuhn:quasi-shuffle}, Kuhn showed that for every finitely generated $\DRing/(s)$-module $M$, there exists a Hopf algebra $H(M)$ with indecomposables $M$, and in \cite{bauer:p-polar-iterated-loops}, the author showed that $H(M)$ is free on a $p$-polar coalgebra and (noncanonically) unique with this property. Moreover, the construction $H$ sends direct sums to tensor products.
	Indecomposable $\DRing/(s)$-modules are suspensions of modules of the form $k[t]/(t^{n+1})$, with corresponding Hopf algebra isomorphic to $H(\Gamma_n)$ with $\Gamma_n$ as displayed in Fig.~\ref{fig:cofreeonfinite} for $n=4$, or of the form $k[t,t^{-1}]/k[t]$, with corresponding Hopf algebra $H(\Gamma_\infty)$ as displayed in Fig.~\ref{fig:cofreeoninfinite}.
	\item This is well-known, e.g. by \cite[Th\'eor\`eme 3.2]{schoeller:Hopf}.
\end{enumerate}
\end{proof}

\section{Some examples}

In this section, we collect a few examples that show that our results are sharp in many respects.

\begin{example}\label{ex:gradingisessential}
This example is to show that the grading is crucial and that, in particular, the connectedness assumption on $H$ cannot be dropped.

Suppose $H$ is an ungraded abelian Hopf algebra over a perfect field $k$. By \cite{fontaine:groupes-divisibles}, $H$ naturally splits into $H^u \otimes H^m$, where $H^u$ denotes the unipotent part (also known as conilpotent) and $H^m$ is the part of multiplicative type. The latter is classified by an abelian group $A$ with a continuous action of the absolute Galois group $\Gamma$ of $k$; if the $\Gamma$-action is trivial then the associated Hopf algebra is just the group ring $k[A]$. So a classification of $H^m$, under some finiteness assumptions, is feasible. Note that connected, graded Hopf algebras are automatically unipotent for degree reasons.

Dieudonné theory works as expected for ungraded Hopf algebras (at least when $H^m$ represents a functor taking values in abelian $p$-groups), and unipotent Hopf algebras correspond to Dieudonné modules on which $V$ acts nilpotently. Let $W$ be an $n$-dimensional $\F_p$-vector space, $\phi\colon W \to W$ an endomorphism, and construct a Dieudonné module $M_\phi=W \oplus W$ with $V(x,y) = (0,x)$ and $F(x,y) = (0,\phi(x))$.
Then it is easy to see that $M_\phi \cong M_{\psi}$ if and only if $\phi$ and $\psi$ are conjugate, and $M_\phi$ is decomposable iff $W$ has a $\phi$-invariant direct sum decomposition. So already this finite-dimensional example shows that the moduli of indecomposable modules is rather large and, in an imprecise sense, positive-dimensional (it grows with the size of $k$).
\end{example}


\begin{example}\label{ex:nonpureinj:cont}
We revisit Example~\ref{ex:nonpureinjective} with the non-pure-injective submodule $NPI(\sigma)$ of the locally pure-injective module $LPI(\sigma)$. The inclusion $NPI(\sigma) \to LPI(\sigma)$ is pure by Lemma~\ref{lemma:puritycharacterization}. The quotient, $LPI(\sigma)/NPI(\sigma)$, is the pure-injective string module $M(\sigma|_1^\infty)$.

I do not know of an example of an $R$-module of local pure-injective dimension greater than $1$.
\end{example}

\begin{example}\label{ex:largeindecomposables}
In this example, we will construct arbitrarily large indecomposable $\DRing$-modules of finite type (and hence Hopf algebras), showing in particular that non-$p$-torsion Hopf algebras do not decompose into tensor products of basic Hopf algebras.

Let $\Gamma_i$ be the basic Hopf graph
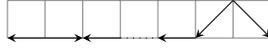
\[
\begin{tikzpicture}[scale=0.5]
\draw[gray,very thin] (0,0) grid (7,1);
\draw[stealth-] (0,0) -- (1,0);
\draw[-stealth] (1,0) -- (2,0);
\draw[stealth-] (2,0) -- (3,0);
\draw[dotted] (3,0) -- (4,0);
\draw[stealth-] (4,0) -- (5,0);
\draw[stealth-] (5,0) -- (6,1);
\draw[-stealth] (6,1) -- (7,0);
\end{tikzpicture}
\]
with $m = i+4$.

One checks readily that there are no nontrivial maps $M(\Gamma_i) \to M(\Gamma_j)$ unless $i=j$.

There is a unique nontrivial extension $0 \to \Sigma k \to M_i \to M(\Gamma_i) \to 0$ for each $i$, resulting in a basic Hopf algebra with graph
\[
\begin{tikzpicture}[scale=0.5]
\draw[gray,very thin] (0,0) grid (7,1);
\draw[stealth-] (0,0) -- (1,1);
\draw[-stealth] (1,1) -- (2,0);
\draw[stealth-] (2,0) -- (3,0);
\draw[dotted] (3,0) -- (4,0);
\draw[stealth-] (4,0) -- (5,0);
\draw[stealth-] (5,0) -- (6,1);
\draw[-stealth] (6,1) -- (7,0);
\end{tikzpicture}
\]
Thus $\Ext^1(M(\Gamma_i),\Sigma k) \cong k\langle \alpha_i\rangle$. For any $N>0$, the class
\[
(\alpha_1,\dots,\alpha_N) \in \Ext^1\Bigl(\bigoplus_{i=1}^N M(\Gamma_i),\Sigma k\Bigr) \cong \bigoplus_{i=1}^N \Ext^1(M(\Gamma_i),\Sigma k)
\]
thus defines an $\DRing$-module $M$. I claim that this module is indecomposable. We use the following result, which is a special case of \cite[Theorem~2.3]{hassler-wiegand:big-indecomposable}:
\begin{lemma}
Let $Q$ be a finitely generated $R$-module of positive depth, $T$ an indecomposable finitely generated $R$-module of finite length, and $\alpha \in \Ext^1(Q,T)$. Suppose that for each $f\in \End(Q)-\{0\}$, $f^*\alpha \neq 0$. Then $\alpha$ represents an indecomposable module.
\end{lemma}
In our case, $Q=\bigoplus_{i=1}^N M(\Gamma_i)$, $T = \Sigma k$, and $\alpha = (\alpha_1,\dots,\alpha_N)$. We have that $\End(Q) \cong k^N$ because there are no nontrivial maps $M(\Gamma_i) \to M(\Gamma_j)$ for $i \neq j$. For $f=(f_1,\dots,f_N) \in \End(Q)$, we have that $f^*\alpha = (f_1\alpha_1,\dots,f_N\alpha_N)$ and since all $\alpha_i$ are nontrivial, $f^*\alpha \neq 0$.
\end{example}

\bibliographystyle{alpha}
\bibliography{bibliography}

@article{los:abelian-groups,
	author = {\L o\'s, J.},
	doi = {10.4064/fm-44-1-84-90},
	fjournal = {Polska Akademia Nauk. Fundamenta Mathematicae},
	issn = {0016-2736,1730-6329},
	journal = {Fund. Math.},
	mrclass = {20.0X},
	mrnumber = {89199},
	mrreviewer = {D.\ K.\ Harrison},
	pages = {84--90},
	title = {Abelian groups that are direct summands of every abelian group which contains them as pure subgroups},
	url = {https://doi.org/10.4064/fm-44-1-84-90},
	volume = {44},
	year = {1957}}

@article{bauer:p-polar-iterated-loops,
	author = {Bauer, Tilman},
	doi = {10.4310/hha.2025.v27.n1.a18},
	fjournal = {Homology, Homotopy and Applications},
	issn = {1532-0073,1532-0081},
	journal = {Homology Homotopy Appl.},
	mrclass = {14L05 (13A35 14L15 14L17 16T05 55P35)},
	mrnumber = {4901653},
	number = {1},
	pages = {355--382},
	title = {Graded {$p$}-polar rings and the homology of {$\Omega^n\Sigma^nX$}},
	url = {https://doi-org.focus.lib.kth.se/10.4310/hha.2025.v27.n1.a18},
	volume = {27},
	year = {2025}}

@article{warfield:purity,
	author = {Warfield, Jr., R. B.},
	fjournal = {Pacific Journal of Mathematics},
	issn = {0030-8730},
	journal = {Pacific J. Math.},
	mrclass = {16.40},
	mrnumber = {242885},
	mrreviewer = {B. Stenstr\"{o}m},
	pages = {699--719},
	title = {Purity and algebraic compactness for modules},
	url = {http://projecteuclid.org/euclid.pjm/1102983324},
	volume = {28},
	year = {1969}}

@article{gabriel:categories-abeliennes,
	author = {Gabriel, Pierre},
	fjournal = {Bulletin de la Soci\'{e}t\'{e} Math\'{e}matique de France},
	issn = {0037-9484},
	journal = {Bull. Soc. Math. France},
	mrclass = {18.20},
	mrnumber = {232821},
	mrreviewer = {T.-Y. Lam},
	pages = {323--448},
	title = {Des cat\'{e}gories ab\'{e}liennes},
	url = {http://www.numdam.org/item?id=BSMF_1962__90__323_0},
	volume = {90},
	year = {1962}}

@book{lam:modules-and-rings,
	author = {Lam, T. Y.},
	doi = {10.1007/978-1-4612-0525-8},
	isbn = {0-387-98428-3},
	mrclass = {16-01},
	mrnumber = {1653294},
	mrreviewer = {Jonathan Golan},
	pages = {xxiv+557},
	publisher = {Springer-Verlag, New York},
	series = {Graduate Texts in Mathematics},
	title = {Lectures on modules and rings},
	url = {https://doi.org/10.1007/978-1-4612-0525-8},
	volume = {189},
	year = {1999}}

@article{crawley-boevey:string-algebras,
	author = {Crawley-Boevey, William},
	doi = {10.1090/tran/7032},
	fjournal = {Transactions of the American Mathematical Society},
	issn = {0002-9947},
	journal = {Trans. Amer. Math. Soc.},
	mrclass = {16D70 (13C05)},
	mrnumber = {3766850},
	mrreviewer = {Shengyong Pan},
	number = {5},
	pages = {3289--3313},
	title = {Classification of modules for infinite-dimensional string algebras},
	url = {https://doi.org/10.1090/tran/7032},
	volume = {370},
	year = {2018}}

@article{touze:exponential,
	author = {Touz\'{e}, Antoine},
	doi = {10.1093/imrn/rnz301},
	fjournal = {International Mathematics Research Notices. IMRN},
	issn = {1073-7928},
	journal = {Int. Math. Res. Not. IMRN},
	mrclass = {18E05 (18A22)},
	mrnumber = {4307690},
	number = {17},
	pages = {13305--13415},
	title = {On the structure of graded commutative exponential functors},
	url = {https://doi.org/10.1093/imrn/rnz301},
	year = {2021}}

@incollection{hassler-wiegand:big-indecomposable,
	author = {Hassler, Wolfgang and Wiegand, Roger},
	booktitle = {Abelian groups, rings, modules, and homological algebra},
	doi = {10.1201/9781420010763.ch15},
	mrclass = {13C14 (13D07 13H10)},
	mrnumber = {2229110},
	mrreviewer = {Santiago Zarzuela Armengou},
	pages = {159--174},
	publisher = {Chapman \& Hall/CRC, Boca Raton, FL},
	series = {Lect. Notes Pure Appl. Math.},
	title = {Big indecomposable mixed modules over hypersurface singularities},
	url = {https://doi.org/10.1201/9781420010763.ch15},
	volume = {249},
	year = {2006}}

@article{wraith:abelian-hopf-algebras,
	author = {Wraith, G. C.},
	doi = {10.1016/0021-8693(67)90001-4},
	fjournal = {Journal of Algebra},
	issn = {0021-8693},
	journal = {J. Algebra},
	mrclass = {18.20},
	mrnumber = {210700},
	mrreviewer = {M. Barr},
	pages = {135--156},
	title = {Abelian {H}opf algebras},
	url = {https://doi.org/10.1016/0021-8693(67)90001-4},
	volume = {6},
	year = {1967}}

@article{hopf:Gruppenmannigfalten,
	author = {Hopf, Heinz},
	doi = {10.2307/1968985},
	fjournal = {Annals of Mathematics. Second Series},
	issn = {0003-486X},
	journal = {Ann. of Math. (2)},
	mrclass = {56.0X},
	mrnumber = {4784},
	mrreviewer = {W. Hurewicz},
	pages = {22--52},
	title = {\"{U}ber die {T}opologie der {G}ruppen-{M}annigfaltigkeiten und ihre {V}erallgemeinerungen},
	url = {https://doi.org/10.2307/1968985},
	volume = {42},
	year = {1941}}

@article{webb:graded-modules,
	author = {Webb, Cary},
	doi = {10.2307/2044864},
	fjournal = {Proceedings of the American Mathematical Society},
	issn = {0002-9939},
	journal = {Proc. Amer. Math. Soc.},
	mrclass = {13C05},
	mrnumber = {792261},
	mrreviewer = {J. Hausen},
	number = {4},
	pages = {565--571},
	title = {Decomposition of graded modules},
	url = {https://doi.org/10.2307/2044864},
	volume = {94},
	year = {1985}}

@incollection{ravenel:dieudonne,
	author = {Ravenel, Douglas C.},
	booktitle = {Conference on homotopy theory ({E}vanston, {I}ll., 1974)},
	mrclass = {55S99 (16A05 16A24)},
	mrnumber = {761728},
	pages = {177--183},
	publisher = {Soc. Mat. Mexicana, M\'{e}xico},
	series = {Notas Mat. Simpos.},
	title = {Dieudonn\'{e} modules for abelian {H}opf algebras},
	volume = {1},
	year = {1975}}

@article{schoeller:Hopf,
	author = {Schoeller, Colette},
	doi = {10.1007/BF01273307},
	fjournal = {Manuscripta Mathematica},
	issn = {0025-2611},
	journal = {Manuscripta Math.},
	mrclass = {13.50 (18.00)},
	mrnumber = {281709},
	mrreviewer = {Richard G. Larson},
	pages = {133--155},
	title = {\'{E}tude de la cat\'{e}gorie des alg\`ebres de {H}opf commutatives connexes sur un corps},
	url = {https://doi.org/10.1007/BF01273307},
	volume = {3},
	year = {1970}}

@article{kuhn:quasi-shuffle,
	author = {Kuhn, Nicholas J.},
	doi = {10.1016/j.aim.2020.107183},
	fjournal = {Advances in Mathematics},
	issn = {0001-8708},
	journal = {Adv. Math.},
	mrclass = {57T05 (05E05 16T05 55P35)},
	mrnumber = {4092985},
	mrreviewer = {Tao Yang},
	pages = {107183, 30},
	title = {Split {H}opf algebras, quasi-shuffle algebras, and the cohomology of {$\Omega \Sigma X$}},
	url = {https://doi.org/10.1016/j.aim.2020.107183},
	volume = {369},
	year = {2020}}

@article{borel:homologie-des-groupes-de-Lie,
	author = {Borel, Armand},
	doi = {10.2307/2372574},
	fjournal = {American Journal of Mathematics},
	issn = {0002-9327},
	journal = {Amer. J. Math.},
	mrclass = {20.0X},
	mrnumber = {64056},
	mrreviewer = {W. T. van Est},
	pages = {273--342},
	title = {Sur l'homologie et la cohomologie des groupes de {L}ie compacts connexes},
	url = {https://doi.org/10.2307/2372574},
	volume = {76},
	year = {1954}}

@article{bauer:p-polar,
	author = {Tilman Bauer},
	journal = {Math. Scand.},
	pages = {35--53},
	title = {Affine and formal abelian group schemes on $p$-polar rings},
	volume = {128},
	year = {2022}}

@book{fontaine:groupes-divisibles,
	author = {Fontaine, Jean-Marc},
	mrclass = {14L05},
	mrnumber = {0498610},
	mrreviewer = {Loren D. Olson},
	note = {Ast{{\'e}}risque, No. 47-48},
	pages = {i+262},
	publisher = {Soci{\'e}t{\'e} Math{\'e}matique de France, Paris},
	title = {Groupes {$p$}-divisibles sur les corps locaux},
	year = {1977}}

@article{milnor-moore:hopf,
	author = {Milnor, John W. and Moore, John C.},
	fjournal = {Annals of Mathematics. Second Series},
	issn = {0003-486X},
	journal = {Ann. of Math. (2)},
	mrclass = {55.40 (55.34)},
	mrnumber = {0174052 (30 \#4259)},
	mrreviewer = {I. M. James},
	pages = {211--264},
	title = {On the structure of {H}opf algebras},
	volume = {81},
	year = {1965}}

@book{adamek-rosicky:presentable-accessible,
	address = {Cambridge},
	author = {Ad{\'a}mek, Ji{\v{r}}{\'{\i}} and Rosick{\'y}, Ji{\v{r}}{\'{\i}}},
	doi = {10.1017/CBO9780511600579},
	isbn = {0-521-42261-2},
	mrclass = {18Axx (18-02)},
	mrnumber = {1294136 (95j:18001)},
	mrreviewer = {J. R. Isbell},
	pages = {xiv+316},
	publisher = {Cambridge University Press},
	series = {London Mathematical Society Lecture Note Series},
	title = {Locally presentable and accessible categories},
	url = {http://dx.doi.org/10.1017/CBO9780511600579},
	volume = {189},
	year = {1994}}

@article{bousfield:p-adic-lambda-rings,
	author = {Bousfield, A. K.},
	coden = {MAZEAX},
	fjournal = {Mathematische Zeitschrift},
	issn = {0025-5874},
	journal = {Math. Z.},
	mrclass = {55S25 (16W30 55N15)},
	mrnumber = {MR1417857 (97i:55033)},
	mrreviewer = {Paul G. Goerss},
	number = {3},
	pages = {483--519},
	title = {On {$p$}-adic {$\lambda$}-rings and the {$K$}-theory of {$H$}-spaces},
	volume = {223},
	year = {1996}}

\end{document}